\documentclass[reqno]{amsart}
\usepackage{amssymb,amsmath,hyperref}
\usepackage{amsrefs, float}
\usepackage[foot]{amsaddr}
\usepackage{bbold,stackrel}
 
\newcommand{\Z}{{\textsf{\textup{Z}}}}

\newtheorem{thm}{Theorem}

\newtheorem{cor}[thm]{Corollary}
\newtheorem{defi}[thm]{Definition}
\newtheorem{rem}[thm]{Remark}
\newtheorem{nota}[thm]{Notation}

\newtheorem{princ}[thm]{Principle}

\newtheorem{ack}[thm]{Acknowledgement}

\newtheorem*{tempo*}{Template}

\newcommand\be{\begin{equation}}
\newcommand\ee{\end{equation}} 

\usepackage{amsmath,amsfonts} 

\usepackage[applemac]{inputenc}

\def\bdefi{\begin{defi}\rm}
\def\edefi{\end{defi}}
\def\bnota{\begin{nota}\rm}
\def\enota{\end{nota}}
\def\FIVE{\Pi_{1}^{1}\text{-\textup{\textsf{CA}}}_{0}}
\def\SIX{\Pi_{2}^{1}\text{-\textsf{\textup{CA}}}_{0}}

\def\SIXK{\Pi_{k}^{1}\text{-\textsf{\textup{CA}}}_{0}^{\omega}}

\def\ZF{\textup{\textsf{ZF}}}

\def\L{\textsf{\textup{L}}}

 \def\r{\mathbb{r}}

\def\RCA{\textup{\textsf{RCA}}}
\def\({\textup{(}}
\def\){\textup{)}}

\def\RCAo{\textup{\textsf{RCA}}_{0}^{\omega}}
\def\ACAo{\textup{\textsf{ACA}}_{0}^{\omega}}
\def\WKL{\textup{\textsf{WKL}}}

\def\bye{\end{document}}
\def\N{{\mathbb  N}}
\def\Q{{\mathbb  Q}}
\def\R{{\mathbb  R}}

\def\SS{\textup{\textsf{S}}}

\def\MUC{\textup{\textsf{MUC}}}

\def\di{\rightarrow}

\def\asa{\leftrightarrow}

\def\ACA{\textup{\textsf{ACA}}}

\def\QFAC{\textup{\textsf{QF-AC}}}
\def\AC{\textup{\textsf{AC}}}

\def\SUBNET{\textup{\textsf{SUBNET}}}

\def\BW{\textup{\textsf{BW}}}

\def\CBN{\textup{\textsf{CB}}\mathbb{N}}
\def\Korg{\textup{\textsf{Korg}}}

\def\cocode{\textup{\textsf{cocode}}}

\def\NIN{\textup{\textsf{NIN}}}
\def\NCC{\textup{\textsf{NCC}}}
\def\NBI{\textup{\textsf{NBI}}}

\def\HBC{\textup{\textsf{HBC}}}

\def\BOOT{\textup{\textsf{BOOT}}}

\def\IND{\textup{\textsf{IND}}}

\def\net{\textup{\textsf{net}}}

\def\lex{\textup{\textsf{lex}}}

\def\BW{\textup{\textsf{BW}}}

\def\fin{\textup{\textsf{fin}}}

\def\CAC{\textup{\textsf{CAC}}}

\def\DCA{\Delta\textup{\textsf{-CA}}}

\def\CAC{\textup{\textsf{CAC}}}

\def\MCT{\textup{\textsf{MCT}}}

\def\eps{\varepsilon}

\def\X{\textup{\textsf{X}}}

\def\ADS{\textup{\textsf{ADS}}}

\def\RT{\textup{\textsf{RT}}}

\def\ECF{\textup{\textsf{ECF}}}

\usepackage{graphicx}
\usepackage{tikz}
\usetikzlibrary{matrix, shapes.misc}

\numberwithin{equation}{section}
\numberwithin{thm}{section}

\usepackage{comment,tikz-cd}

\begin{document}
\title[Countable sets versus sets that are countable]{Countable sets versus sets that are countable in Reverse Mathematics}
\author{Sam Sanders}
\address{Institute for Philosophy II, RUB Bochum, Germany}
\email{sasander@me.com}
\begin{abstract}
The program \emph{Reverse Mathematics} (RM for short) seeks to identify the axioms necessary to prove theorems of ordinary mathematics, usually working in the language of second-order arithmetic $\L_{2}$.
A major theme in RM is therefore the study of structures that are \emph{countable} or can be \emph{approximated} by countable sets.  
Now, countable sets must be represented by \emph{sequences} here, because the higher-order definition of `countable set' involving injections/bijections to $\N$ \emph{cannot} be directly expressed in $\L_{2}$.
Working in Kohlenbach's \emph{higher-order} RM, we investigate various central theorems, e.g.\ those due to K\"onig, Ramsey, Bolzano, Weierstrass, and Borel, in their (often original) formulation involving the definition of `countable set' based on injections/bijections to $\N$.  This study turns out to be closely related to the logical properties of the uncountably of $\R$, recently developed by the author and Dag Normann.   
Now, `being countable' can be expressed by the existence of an injection to $\N$ (Kunen) or the existence of a bijection to $\N$ (Hrbacek-Jech).
The former (and not the latter) choice yields `explosive' theorems, i.e.\ relatively weak statements that become much stronger when combined with discontinuous functionals, even up to $\SIX$.
Nonetheless, replacing `sequence' by `countable set' seriously reduces the first-order strength of these theorems, whatever the notion of `set' used.  
Finally, we obtain `splittings' involving e.g.\ lemmas by K\"onig and theorems from the RM zoo, showing that the latter are `a lot more tame' when formulated with countable sets.
\end{abstract}
\setcounter{page}{0}
\tableofcontents
\thispagestyle{empty}
\newpage

\maketitle
\thispagestyle{empty}
\section{Introduction}\label{intro}
Concepts like `countable subset of $\R$' and `the uncountability of $\R$' involve arbitrary mappings with domain $\R$, and are therefore best studied in a language that has such objects as first-class citizens. 
Obviousness, much more than beauty, is however in the eye of the beholder.  Lest we be misunderstood, we formulate a blanket caveat: all notions (computation, continuity, function, open set, et cetera) used in this paper are to be interpreted via their higher-order definitions, also listed below, \emph{unless explicitly stated otherwise}.    
\subsection{Historical background and motivation}\label{bintro}
In a nutshell, this paper deals with the study of the logical and computational properties of theorems of ordinary\footnote{Simpson describes \emph{ordinary mathematics} in \cite{simpson2}*{I.1} as \emph{that body of mathematics that is prior to or independent of the introduction of abstract set theoretic concepts}.} mathematics formulated using the definition of `countable set' based on injections/bijections to $\N$, in particular when this choice 
results in significant differences compared to the formulation involving sequences.  A more detailed description is in Section \ref{detail}, while we now sketch the (historical) motivation for this paper, based on mathematical household names like Borel, K\"onig, Ramsey, and Cantor, as well as the mathematical mainstream

\smallskip

First of all, the notion of `countable set' can be defined in various ways.  However, it is an empirical observation, witnessed by countless mainstream textbooks, that to show that a set is countable one often \emph{only} 
constructs an injection (or bijection) to $\N$.  
When \emph{given} a countable set, one (additionally) assumes that this set can be \emph{enumerated}, i.e.\ represented by some sequence.  Hence, whatever one's preferred definition of `countable set' may be, implicit in much of mathematical practise is the following most basic principle about countable sets: 
\begin{center}
\emph{a set that can be mapped to $\N$ via an injection \(or bijection\) can be enumerated}.
\end{center}
This basic principle is formalised as $\cocode_{i}$ for $i=0,1$ in Section \ref{bumm}.
We now provide some more historical and conceptual motivation for this study.  

\smallskip

Secondly, Borel formulates the Heine-Borel theorem in \cite{opborrelen2} using \emph{countable collections} of intervals (rather than sequences), i.e.\ the study of countable sets \emph{an sich} has its roots in ordinary mathematics, namely as discussed in the following remark.
\begin{rem}[Borel's Heine-Borel]\label{horgku}\rm
Borel introduces the notion of `countable set' (French: \emph{ensemble d\'enomerable}) via bijections to $\N$ in \cite{opborrelen2}*{p.~6}. 
He then proceeds to explain the provenance of `countable' (French: \emph{d\'enomerable}), namely that the elements of such sets can be enumerated, i.e.\ listed as a sequence.  
In this way, Borel makes use of the principle $\cocode_{1}$ from Section \ref{bumm} which state that certain countable sets can be enumerated.  

\smallskip

Moreover, Borel's formulation of the Heine-Borel theorem in \cite{opborrelen2}*{p.\ 42} involves \emph{une infinit\'e d\'enombrable d'intervalles}, i.e.\ a countable infinity of intervals.
Thus, Borel's proof starts with the following (originally French): 
\begin{quote}
Let us enumerate our intervals, one after the other, according to whatever law, but determined. \cite{opborrelen2}*{p.\ 42}
\end{quote}
This sentence constitutes another use of the aforementioned principle $\cocode_{1}$.
Borel then proceeds with the usual `interval-halving' proof, similar to Cousin in \cite{cousin1}. 
Similar observations can be made for \cite{opborrelen}*{p.\ 51}, where Borel uses language similar to the previous quote. 
\end{rem}
Similar to the previous remark, Ramsey (\cite{keihardrammen}*{p.\ 264}) and K\"onig (\cite{koning147}) used set theoretic jargon to formulate their eponymous theorem and lemma.  K\"onig even explicitly studies countable sets in \cite{koning147}, while Ramsey only mentions the distinction between the finite and infinite.  
All these are well-studied in \emph{Reverse Mathematics} (RM hereafter; see Section~\ref{prelim1}) with countable sets formulated \emph{using sequences}, and it is therefore a natural question what happens if we work with countable sets \emph{involving injections or bijections to $\N$} instead.  In this paper, we provide a partial answer to this question, which constitutes a contribution to Kohlenbach's \emph{higher-order} RM (see Section \ref{prelim1}).
We note that the (second-order) concept of `countable set' is introduced in \cite{simpson2}*{V.4.2} and used throughout \cite{simpson2}.

\smallskip

Thirdly, more historical motivation is provided by \emph{the uncountability of $\R$} which has an elegant formulation in terms of countable sets based on Cantor's theorem, as follows.  Now, Cantor's \emph{first} set theory paper \cite{cantor1}, published in 1874 and boasting its own Wikipedia page (\cite{wica}), establishes the uncountability of $\R$ as a corollary to:  
\begin{center}
\emph{for any sequence of reals, there is another real different from all reals in the sequence}.  
\end{center}
The logical and computational properties of this theorem, called \emph{Cantor's theorem}, are well-known: it is provable in weak and constructive systems (\cite{simpson2}*{II.4.9} and \cite{bish1}*{p.\ 25}) while there is an efficient computer program 
that computes the real in the conclusion from the sequence (\cite{grayk}).  By contrast, the uncountability of $\R$ has only recently been studied in detail (\cite{dagsamX}) in the guise of the following principles:
\begin{itemize}
\item $\NIN$: \emph{there is no injection from $[0,1]$ to $\N$},
\item  $\NBI$: \emph{there is no bijection from $[0,1]$ to $\N$},
\end{itemize}
Interpreting `countable set' as `there is an injection to $\N$' (see Definitions~\ref{openset} and~\ref{standard}), $\NIN$ is equivalent to the following reformulation of Cantor's theorem:
\be\tag{\textsf{A}}\label{hong}
\textup{\emph{for countable $A\subset \R$, there is a real $y\in [0,1]$ different from all reals in $A$}.}
\ee
Moreover, $\NIN$ follows from Cantor's \cite{cantorb}*{\S16, Theorem A*)} restricted to $\R$:
\be\tag{\textsf{B}}\label{hong2}
\textup{\emph{If a subset $A\subset [0,1]$ is countable, then $A$ cannot be perfect}}.
\ee
A second-order version of \eqref{hong2} is provable in the base theory of RM (\cite{simpson2}*{II.5.9}).
The same results hold for bijections and $\NBI$ \emph{mutatis mutandis} by Theorem \ref{momo}.  In this light, the study of countable sets has its roots (implicitly and explicitly) in the work of Borel and Cantor (and others), and their (semi-)eponymous theorems.

\smallskip

Fourth, we provide some motivation based on mathematical logic.  While \emph{prima facie} quite similar, Cantor's theorem and \eqref{hong}, have \emph{hugely} different logical and computational properties.   
Indeed, as noted above, there are proofs of Cantor's theorem in weak and constructive systems, while the real claimed to exist can be computed efficiently.  
By contrast, $\NIN$ and $\NBI$ are \emph{hard to prove}\footnote{The \textbf{third-order} statements $\NIN$ and $\NBI$ are not provable in $\Z_{2}^{\omega}$, 
a conservative and essentially \textbf{third-order} extension of second-order arithmetic $\Z_{2}$ going back to Sieg-Feferman (\cite{boekskeopendoen}*{p.~129}), i.e.\ we can say that $\NIN$ and \eqref{hong} are \emph{hard to prove}.
The reals claimed to exist by $\NIN$ are similarly \emph{hard to compute}, in the sense of Kleene's S1-S9.  All these results are proved in \cite{dagsamX}*{\S3-4}, while the associated definitions may (also) be found in Section \ref{HCT} of this paper.} in that full second-order arithmetic comes to the fore.  
The real $y$ from \eqref{hong} is similarly hard to compute in terms of the data, within Kleene's higher-order framework.  
Since \eqref{hong} is so elementary, these observations suggest that theorems about countable sets have rather extreme/interesting logical and computational properties.

\smallskip

Fifth, the aforementioned properties of $\NIN$ only serve to motivate the goal of this paper, namely the study of theorems of ordinary mathematics formulated using the definition of `countable set' involving injections/bijections to $\N$, in particular when this choice 
results in significant differences compared to the formulation involving sequences.  The correct (or at least broader) interpretation of the logical and computational properties of $\NIN$ and $\NBI$ may be found in \cite{dagsamX}*{\S1-2}.
We will say that, by \cite{dagsamX}*{Theorem 3.2}, the principles $\NIN$ and $\NBI$ are provable \emph{without using the Axiom of Choice} while the same holds for the theorems considered in this paper, by the below.  

\smallskip

Finally, this paper (clearly) constitutes a spin-off from \cite{dagsamX}, which is part of our joint project with Dag Normann on the logical and computational properties of the uncountable. 
The interested reader may consult \cite{dagsamIII, dagsamX} as an introduction.  

\subsection{Countable infinity and Reverse Mathematics}\label{detail}
We discuss the role of the countable in (Reverse) Mathematics and formulate four questions (Q0)-(Q3) that will be given (partial) answers in this paper. 

\smallskip

Now, the word `countable' and variations occur hundreds of times in Simpson's excellent monograph \cite{simpson2}, the unofficial bible of RM, and in \cite{simpson1}, a sequel consisting of RM papers from 2001.
The tally for Hirschfeldt's monograph \cite{dsliceke} on the RM zoo (see \cite{damirzoo} for the latter) is about one hundred.  
Countable infinity does indeed take centre stage, as is clear from the following quotes by Hirschfeldt and Simpson.
\begin{quote}
We finish this section with an important remark: The approaches to analyzing the strength of theorems we will discuss here are tied to the countably infinite. (\cite{dsliceke}*{p.\ 6})
\end{quote}
\begin{quote}
Since in ordinary mathematics the objects studied are almost always
countable or separable, it would seem appropriate to consider a language
in which countable objects occupy center stage. (\cite{simpson2}*{p.~2})
\end{quote}
As detailed in Section \ref{prelim1}, RM studies ordinary mathematics and does so using the language of second-order arithmetic $\L_{2}$.  
Now, the definition of `countable subset of $\R$' based on injections/bijections to $\N$ cannot be expressed in $\L_{2}$.  Indeed, countable
sets are given by \emph{sequences} in RM, namely as in \cite{simpson2}*{V.4.2}. 
It is therefore a natural question what happens to the results of RM if we use the definition of countable set involving injections or bijections to $\N$ (see Definition~\ref{standard}), say working in Kohlenbach's \emph{higher-order} RM (see Section \ref{prelim1}).
This paper is devoted to the study of this question, as sketched in the rest of this section, right after the following quote providing ample motivation for this study.  
\begin{quote}
This situation has prompted some authors, for example
Bishop/Bridges \cite{bridge1}*{p.\ 38}, to build a modulus of uniform continuity into their definitions of continuous function.
Such a procedure may be appropriate for
Bishop since his goal is to replace ordinary mathematical theorems by
their ``constructive'' counterparts.
However, as explained in chapter I, our goal is quite different.
Namely, we seek to draw out the set existence
assumptions which are implicit in the ordinary mathematical theorems \emph{as they stand}. (\cite{simpson2}*{p.\ 137})
\end{quote}
As noted above, Borel (\cite{opborrelen2}) and K\"onig (\cite{koning147}), formulate their (semi-) eponymous theorems, namely the \emph{Heine-Borel theorem} and \emph{K\"onig's \(infinity\) lemma}, using the definition of countable sets based on bijections to $\N$, while Ramsey formulates \emph{Ramsey's theorem} in \cite{keihardrammen}*{p.\ 264} in set theory lingo.  Motivated by Simpson's above quote, we will study versions of the aforementioned theorems formulated using the definition of `countable set' based on injections/bijections to $\N$.

\smallskip

Now, there are (at least) two possible definitions of `countable set', namely based on \emph{injections to $\N$} (Kunen \cite{kunen}) and \emph{bijections to $\N$} (Hrbacek-Jech \cite{hrbacekjech}).  
Hereafter, we shall \textbf{always} refer to the former as `countable' and to the latter as `strongly countable', as laid down in Definition \ref{standard}.  This naming is merely a matter of convenience: we do not claim that injections or bijections to $\N$ constitute the `standard' or `mainstream' definition of countable set.  We discuss this and related topics in Section \ref{bauer}, where we also discuss the grand(er) scheme of things. 

\smallskip

On a related note, we have studied the properties of \emph{nets} in \cite{samcie19, samwollic19,samnetspilot}.  Since nets are the generalisation of sequences to (possibly) uncountable index sets, it is a natural question whether there is a difference between nets with countable index sets on one hand, and sequences on the other hand.  
As noted in \cite{boermet}, Dieudonn\'e formulates a (rather abstract) theorem in \cite{godknows} that is \emph{true} for sequences but \emph{false} for nets indexed by countable index sets.

\smallskip
\noindent
Thus, we are led to the following motivating questions.
\begin{enumerate}
\item[(Q0)] Does replacing `sequence' by `countable set' make a (big) difference? 
\item[(Q1)] Does replacing `countable' by `strongly countable' make a (big) difference? 
\item[(Q2)] Are there (elementary) differences between sequences and nets with countable index sets?
\item[(Q3)] Are there natural splittings\footnote{A relatively rare phenomenon in second-order RM is when a natural theorem $A$ can be `split' as follows into two (somewhat) natural components $B$ and $C$ (say over $\RCA_{0}$):  $A\asa [B\wedge C] $.  As explored in \cite{samsplit}, there is a plethora of splittings to be found in \emph{third-order} RM.\label{lak}} involving `countable' and `strongly countable'?
\end{enumerate}
Regarding (Q0), we exhibit numerous theorems for which the logical strength changes (sometimes dramatically) upon replacing `sequence' by `countable set'.  
This includes theorems by Borel, Ramsey, and K\"onig that were \emph{originally} formulated using `countable set' (rather than `sequence') or at least in set theory lingo.   
Moreover, we identify a number of theorems about countable sets that give rise to $\SIX$ when combined with higher-order $\FIVE$, i.e.\ the Suslin functional.   
The first such example, namely the Bolzano-Weierstrass theorem for countable sets in $2^{\N}$, was identified in \cite{dagsamX}, and we obtain a number of interesting equivalences in Section \ref{LP}.
According to Rathjen (\cite{rathjenICM}), $\SIX$ \emph{dwarfs} $\FIVE$ 
and Martin-L\"of talks of a \emph{chasm} and \emph{abyss} between these two systems in \cite{loefenlei}. 

\smallskip

Regarding (Q1), the previous paragraph establishes that basic well-known theorems from second-order RM formulated with Kunen's notion of countable are \emph{explosive}, i.e.\
become much stronger when combined with discontinuous (comprehension) functionals.  
These theorems are not provable in $\Z_{2}^{\omega}$, while using the Hrbacek-Jech notion of countable set, the same theorems are \emph{no longer} explosive, while still not provable in $\Z_{2}^{\omega}$.
All these theorems are provable in $\Z_{2}^{\Omega}$; $\Z_{2}^{\omega}$ and $\Z_{2}^{\Omega}$ are the conservative extensions of $\Z_{2}$ with third and fourth-order comprehension from Section~\ref{HCT}.

\smallskip

As to (Q2), nets with countable index sets give rise to explosive (convergence) theorems by Theorem \ref{angel}, i.e.\ the power of nets does not (fully) depend on the cardinality of the index set. 
In particular, the aforementioned convergence theorems exist at the level of $\RCA_{0}$, but yield $\SIX$ when combined with higher-order $\FIVE$, i.e.\ the Suslin functional.
Nets with \emph{uncountable} index sets, namely $\N^{\N}$, are \emph{similarly} explosive when combined with the Suslin functional by \cite{samph}*{\S3}.  
Hence, nets with countable index sets can behave \emph{quite differently} compared to sequences.  

\smallskip

Regarding (Q3), we obtain some nice \emph{splittings} involving various theorems as follows: convergence theorems for nets with (strongly) countable index sets, lemmas by K\"onig (\cite{koning147}), theorems from the RM zoo (\cite{damirzoo}), and various coding principles that connect countable and strongly countable sets and/or sequences, including the Cantor-Bernstein theorem for $\N$.  As noted in Footnote \ref{lak}, a `splitting' is an equivalence $A\asa [B\wedge C]$ where a natural theorem $A$ can be `spit' into two independent (somewhat) natural parts $B$ and $C$.  
These results suggest that principles from the RM zoo lose their `exceptional' behaviour when formulated with countable sets instead of sequences.  

\section{Preliminaries}
We introduce \emph{Reverse Mathematics} in Section \ref{prelim1}, as well as Kohlebach's generalisation to \emph{higher-order arithmetic}, and the associated base theory $\RCAo$.  
We introduce some essential axioms in Section~\ref{HCT}.  

\subsection{Reverse Mathematics}\label{prelim1}
Reverse Mathematics is a program in the foundations of mathematics initiated around 1975 by Friedman (\cites{fried,fried2}) and developed extensively by Simpson (\cite{simpson2}).  
The aim of RM is to identify the minimal axioms needed to prove theorems of ordinary, i.e.\ non-set theoretical, mathematics. 

\smallskip

We refer to \cite{stillebron} for a basic introduction to RM and to \cite{simpson2, simpson1} for an overview of RM.  We expect familiarity with RM, but do sketch some aspects of Kohlenbach's \emph{higher-order} RM (\cite{kohlenbach2}) essential to this paper, including the base theory $\RCAo$ (Definition \ref{kase}).  
As will become clear, the latter is officially a type theory but can accommodate (enough) set theory via Definitions \ref{openset} and \ref{standard}. 

\smallskip

First of all, in contrast to `classical' RM based on \emph{second-order arithmetic} $\Z_{2}$, higher-order RM uses $\L_{\omega}$, the richer language of \emph{higher-order arithmetic}.  
Indeed, while the former is restricted to natural numbers and sets of natural numbers, higher-order arithmetic can accommodate sets of sets of natural numbers, sets of sets of sets of natural numbers, et cetera.  
To formalise this idea, we introduce the collection of \emph{all finite types} $\mathbf{T}$, defined by the two clauses:
\begin{center}
(i) $0\in \mathbf{T}$   and   (ii)  If $\sigma, \tau\in \mathbf{T}$ then $( \sigma \di \tau) \in \mathbf{T}$,
\end{center}
where $0$ is the type of natural numbers, and $\sigma\di \tau$ is the type of mappings from objects of type $\sigma$ to objects of type $\tau$.
In this way, $1\equiv 0\di 0$ is the type of functions from numbers to numbers, and  $n+1\equiv n\di 0$.  Viewing sets as given by characteristic functions, we note that $\Z_{2}$ only includes objects of type $0$ and $1$.    

\smallskip

Secondly, the language $\L_{\omega}$ includes variables $x^{\rho}, y^{\rho}, z^{\rho},\dots$ of any finite type $\rho\in \mathbf{T}$.  Types may be omitted when they can be inferred from context.  
The constants of $\L_{\omega}$ include the type $0$ objects $0, 1$ and $ <_{0}, +_{0}, \times_{0},=_{0}$  which are intended to have their usual meaning as operations on $\N$.
Equality at higher types is defined in terms of `$=_{0}$' as follows: for any objects $x^{\tau}, y^{\tau}$, we have
\be\label{aparth}
[x=_{\tau}y] \equiv (\forall z_{1}^{\tau_{1}}\dots z_{k}^{\tau_{k}})[xz_{1}\dots z_{k}=_{0}yz_{1}\dots z_{k}],
\ee
if the type $\tau$ is composed as $\tau\equiv(\tau_{1}\di \dots\di \tau_{k}\di 0)$.  
Furthermore, $\L_{\omega}$ also includes the \emph{recursor constant} $\mathbf{R}_{\sigma}$ for any $\sigma\in \mathbf{T}$, which allows for iteration on type $\sigma$-objects as in the special case \eqref{special}.  Formulas and terms are defined as usual.  
One obtains the sub-language $\L_{n+2}$ by restricting the above type formation rule to produce only type $n+1$ objects (and related types of similar complexity).        
\bdefi\label{kase} 
The base theory $\RCAo$ consists of the following axioms.
\begin{enumerate}
 \renewcommand{\theenumi}{\alph{enumi}}
\item  Basic axioms expressing that $0, 1, <_{0}, +_{0}, \times_{0}$ form an ordered semi-ring with equality $=_{0}$.
\item Basic axioms defining the well-known $\Pi$ and $\Sigma$ combinators (aka $K$ and $S$ in \cite{avi2}), which allow for the definition of \emph{$\lambda$-abstraction}. 
\item The defining axiom of the recursor constant $\mathbf{R}_{0}$: for $m^{0}$ and $f^{1}$: 
\be\label{special}
\mathbf{R}_{0}(f, m, 0):= m \textup{ and } \mathbf{R}_{0}(f, m, n+1):= f(n, \mathbf{R}_{0}(f, m, n)).
\ee
\item The \emph{axiom of extensionality}: for all $\rho, \tau\in \mathbf{T}$, we have:
\be\label{EXT}\tag{$\textsf{\textup{E}}_{\rho, \tau}$}  
(\forall  x^{\rho},y^{\rho}, \varphi^{\rho\di \tau}) \big[x=_{\rho} y \di \varphi(x)=_{\tau}\varphi(y)   \big].
\ee 
\item The induction axiom for quantifier-free formulas of $\L_{\omega}$.
\item $\QFAC^{1,0}$: the quantifier-free Axiom of Choice as in Definition \ref{QFAC}.
\end{enumerate}
\edefi
\noindent
Note that variables (of any finite type) are allowed in quantifier-free formulas of the language $\L_{\omega}$: only quantifiers are banned.
\bdefi\label{QFAC} The axiom $\QFAC$ consists of the following for all $\sigma, \tau \in \textbf{T}$:
\be\tag{$\QFAC^{\sigma,\tau}$}
(\forall x^{\sigma})(\exists y^{\tau})A(x, y)\di (\exists Y^{\sigma\di \tau})(\forall x^{\sigma})A(x, Y(x)),
\ee
for any quantifier-free formula $A$ in the language of $\L_{\omega}$.
\edefi
As discussed in \cite{kohlenbach2}*{\S2}, $\RCAo$ and $\RCA_{0}$ prove the same sentences `up to language' as the latter is set-based and the former function-based.  Recursion as in \eqref{special} is called \emph{primitive recursion}; the class of functionals obtained from $\mathbf{R}_{\rho}$ for all $\rho \in \mathbf{T}$ is called \emph{G\"odel's system $T$} of all (higher-order) primitive recursive functionals.  

\smallskip

We use the usual notations for natural, rational, and real numbers, and the associated functions, as introduced in \cite{kohlenbach2}*{p.\ 288-289}.  
\begin{defi}[Real numbers and related notions in $\RCAo$]\label{keepintireal}\rm~
\begin{enumerate}
 \renewcommand{\theenumi}{\alph{enumi}}
\item Natural numbers correspond to type zero objects, and we use `$n^{0}$' and `$n\in \N$' interchangeably.  Rational numbers are defined as signed quotients of natural numbers, and `$q\in \Q$' and `$<_{\Q}$' have their usual meaning.    
\item Real numbers are coded by fast-converging Cauchy sequences $q_{(\cdot)}:\N\di \Q$, i.e.\  such that $(\forall n^{0}, i^{0})(|q_{n}-q_{n+i}|<_{\Q} \frac{1}{2^{n}})$.  
We use Kohlenbach's `hat function' from \cite{kohlenbach2}*{p.\ 289} to guarantee that every $q^{1}$ defines a real number.  
\item We write `$x\in \R$' to express that $x^{1}:=(q^{1}_{(\cdot)})$ represents a real as in the previous item and write $[x](k):=q_{k}$ for the $k$-th approximation of $x$.    
\item Two reals $x, y$ represented by $q_{(\cdot)}$ and $r_{(\cdot)}$ are \emph{equal}, denoted $x=_{\R}y$, if $(\forall n^{0})(|q_{n}-r_{n}|\leq {2^{-n+1}})$. Inequality `$<_{\R}$' is defined similarly.  
We sometimes omit the subscript `$\R$' if it is clear from context.           
\item Functions $F:\R\di \R$ are represented by $\Phi^{1\di 1}$ mapping equal reals to equal reals, i.e.\ extensionality as in $(\forall x , y\in \R)(x=_{\R}y\di \Phi(x)=_{\R}\Phi(y))$.\label{EXTEN}
\item The relation `$x\leq_{\tau}y$' is defined as in \eqref{aparth} but with `$\leq_{0}$' instead of `$=_{0}$'.  Binary sequences are denoted `$f^{1}, g^{1}\leq_{1}1$', but also `$f,g\in C$' or `$f, g\in 2^{\N}$'.  Elements of Baire space are given by $f^{1}, g^{1}$, but also denoted `$f, g\in \N^{\N}$'.
\item For a binary sequence $f^{1}$, the associated real in $[0,1]$ is $\r(f):=\sum_{n=0}^{\infty}\frac{f(n)}{2^{n+1}}$.\label{detrippe}
\end{enumerate}
\end{defi}
\noindent
Next, we mention the highly useful $\ECF$-interpretation. 
\begin{rem}[The $\ECF$-interpretation]\label{ECF}\rm
The (rather) technical definition of $\ECF$ may be found in \cite{troelstra1}*{p.\ 138, \S2.6}.
Intuitively, the $\ECF$-interpretation $[A]_{\ECF}$ of a formula $A\in \L_{\omega}$ is just $A$ with all variables 
of type two and higher replaced by type one variables ranging over so-called `associates' or `RM-codes' (see \cite{kohlenbach4}*{\S4} and Remark \ref{frorem}); the latter are (countable) representations of continuous functionals.  
The $\ECF$-interpretation connects $\RCAo$ and $\RCA_{0}$ (see \cite{kohlenbach2}*{Prop.\ 3.1}) in that if $\RCAo$ proves $A$, then $\RCA_{0}$ proves $[A]_{\ECF}$, again `up to language', as $\RCA_{0}$ is 
formulated using sets, and $[A]_{\ECF}$ is formulated using types, i.e.\ using type zero and one objects.  
\end{rem}
In light of the widespread use of codes in RM and the common practise of identifying codes with the objects being coded, it is no exaggeration to refer to $\ECF$ as the \emph{canonical} embedding of higher-order into second-order arithmetic. 
For completeness, we list the following notational convention for finite sequences.  
\begin{nota}[Finite sequences]\label{skim}\rm
We assume a dedicated type for `finite sequences of objects of type $\rho$', namely $\rho^{*}$, which we shall only use for $\rho=0,1$.
Since the usual coding of pairs of numbers goes through in $\RCAo$, we shall not always distinguish between $0$ and $0^{*}$. 
Similarly, we do not always distinguish between `$s^{\rho}$' and `$\langle s^{\rho}\rangle$', where the former is `the object $s$ of type $\rho$', and the latter is `the sequence of type $\rho^{*}$ with only element $s^{\rho}$'.  The empty sequence for the type $\rho^{*}$ is denoted by `$\langle \rangle_{\rho}$', usually with the typing omitted.  

\smallskip

Furthermore, we denote by `$|s|=n$' the length of the finite sequence $s^{\rho^{*}}=\langle s_{0}^{\rho},s_{1}^{\rho},\dots,s_{n-1}^{\rho}\rangle$, where $|\langle\rangle|=0$, i.e.\ the empty sequence has length zero.
For sequences $s^{\rho^{*}}, t^{\rho^{*}}$, we denote by `$s*t$' the concatenation of $s$ and $t$, i.e.\ $(s*t)(i)=s(i)$ for $i<|s|$ and $(s*t)(j)=t(j-|s|)$ for $|s|\leq j< |s|+|t|$. For a sequence $s^{\rho^{*}}$, we define $\overline{s}N:=\langle s(0), s(1), \dots,  s(N-1)\rangle $ for $N^{0}<|s|$.  
For a sequence $\alpha^{0\di \rho}$, we also write $\overline{\alpha}N=\langle \alpha(0), \alpha(1),\dots, \alpha(N-1)\rangle$ for \emph{any} $N^{0}$.  By way of shorthand, 
$(\forall q^{\rho}\in Q^{\rho^{*}})A(q)$ abbreviates $(\forall i^{0}<|Q|)A(Q(i))$, which is (equivalent to) quantifier-free if $A$ is.   
\end{nota}

\subsection{Some axioms of higher-order Reverse Mathematics}\label{HCT}
We introduce some axioms of higher-order RM which will be used below.
In particular, we introduce some functionals which constitute the counterparts of second-order arithmetic $\Z_{2}$, and some of the Big Five systems, in higher-order RM.
We use the `standard' formulation from \cite{kohlenbach2, dagsamIII}.  

\smallskip
\noindent
First of all, $\ACA_{0}$ is readily derived from:
\begin{align}\label{mu}\tag{$\mu^{2}$}
(\exists \mu^{2})(\forall f^{1})\big[ (\exists n)(f(n)=0) \di \big(f(\mu(f))=0&\wedge (\forall i<\mu(f))(f(i)\ne 0) \\
& \wedge  (\forall n)(f(n)\ne0)\di   \mu(f)=0\big)    \big], \notag
\end{align}
and $\ACA_{0}^{\omega}\equiv\RCAo+(\mu^{2})$ proves the same sentences as $\ACA_{0}$ by \cite{hunterphd}*{Theorem~2.5}.   The (unique) functional $\mu^{2}$ in $(\mu^{2})$ is also called \emph{Feferman's $\mu$} (\cite{avi2}), 
and is clearly \emph{discontinuous} at $f=_{1}11\dots$; in fact, $(\mu^{2})$ is equivalent to the existence of $F:\R\di\R$ such that $F(x)=1$ if $x>_{\R}0$, and $0$ otherwise (\cite{kohlenbach2}*{\S3}), and to 
\be\label{muk}\tag{$\exists^{2}$}
(\exists \varphi^{2}\leq_{2}1)(\forall f^{1})\big[(\exists n)(f(n)=0) \asa \varphi(f)=0    \big]. 
\ee
\noindent
Secondly, $\FIVE$ is readily derived from the following sentence:
\be\tag{$\SS^{2}$}
(\exists\SS^{2}\leq_{2}1)(\forall f^{1})\big[  (\exists g^{1})(\forall n^{0})(f(\overline{g}n)=0)\asa \SS(f)=0  \big], 
\ee
and $\FIVE^{\omega}\equiv \RCAo+(\SS^{2})$ proves the same $\Pi_{3}^{1}$-sentences as $\FIVE$ by \cite{yamayamaharehare}*{Theorem 2.2}.   The (unique) functional $\SS^{2}$ in $(\SS^{2})$ is also called \emph{the Suslin functional} (\cite{kohlenbach2}).
By definition, the Suslin functional $\SS^{2}$ can decide whether a $\Sigma_{1}^{1}$-formula as in the left-hand side of $(\SS^{2})$ is true or false.  

\smallskip

We similarly define the functional $\SS_{k}^{2}$ which decides the truth or falsity of $\Sigma_{k}^{1}$-formulas; we also define 
the system $\SIXK$ as $\RCAo+(\SS_{k}^{2})$, where  $(\SS_{k}^{2})$ expresses that $\SS_{k}^{2}$ exists.  Note that we allow formulas with \emph{function} parameters, but \textbf{not} \emph{functionals} here.
In fact, Gandy's \emph{Superjump} (\cite{supergandy}) constitutes a way of extending $\FIVE^{\omega}$ to parameters of type two.  
We identify the functionals $\exists^{2}$ and $\SS_{0}^{2}$ and the systems $\ACAo$ and $\SIXK$ for $k=0$.
We note that the operators $\nu_{n}$ from \cite{boekskeopendoen}*{p.\ 129} are essentially $\SS_{n}^{2}$ strengthened to return a witness (if existant) to the $\Sigma_{n}^{1}$-formula at hand.  

\smallskip

\noindent
Thirdly, full second-order arithmetic $\Z_{2}$ is readily derived from $\cup_{k}\SIXK$, or from:
\be\tag{$\exists^{3}$}
(\exists E^{3}\leq_{3}1)(\forall Y^{2})\big[  (\exists f^{1})(Y(f)=0)\asa E(Y)=0  \big], 
\ee
and we therefore define $\Z_{2}^{\Omega}\equiv \RCAo+(\exists^{3})$ and $\Z_{2}^\omega\equiv \cup_{k}\SIXK$, which are conservative over $\Z_{2}$ by \cite{hunterphd}*{Cor.\ 2.6}. 
Despite this close connection, $\Z_{2}^{\omega}$ and $\Z_{2}^{\Omega}$ can behave quite differently, as discussed in e.g.\ \cite{dagsamIII}*{\S2.2}.   The functional from $(\exists^{3})$ is also called `$\exists^{3}$', and we use the same convention for other functionals.  

\smallskip

Fourth, a number of higher-order axioms are studied in \cite{samph} including the following comprehension axiom (see also Remark \ref{hist}):
\be\tag{$\BOOT$}
(\forall Y^{2})(\exists X\subset \N)\big(\forall n\in \N)(n\in X\asa (\exists f \in \N^{\N})(Y(f, n)=0)\big).
\ee
We mention that this axiom is equivalent to e.g.\ the monotone convergence theorem for nets indexed by Baire space (see \cite{samph}*{\S3}).  
The axiom $\BOOT^{-}$ results from restricting $\BOOT$ to functionals $Y$ such that 
\be\label{uneek}
(\forall n\in \N)(\exists \textup{ at most one } f \in \N^{\N})(Y(f, n)=0).
\ee
The weaker $\BOOT^{-}$ appears prominently in the RM-study of open sets given as (third-order) characteristic functions (\cite{dagsamVII}).
In turn, $\BOOT_{C}^{-}$ is $\BOOT^{-}$ with `$\N^{\N}$' replaced by `$2^{\N}$' everywhere; $\BOOT^{-}_{C}$ was introduced in \cite{dagsamX}*{\S3.1} in the study of the Bolzano-Weierstrass theorem for countable sets in Cantor space.

\smallskip

Another weakening of $\BOOT$ is the following axiom, central to \cite{samFLO2, samrecount}.  
\begin{princ}[$\DCA$]\label{DAAS}
For $i=0, 1$, $Y_{i}^{2}$, and $A_i(n)\equiv (\exists f \in \N^\N)(Y_{i}(f,n)=0)$\textup{:}
\[
(\forall n\in \N)(A_0(n) \asa \neg A_1(n))\di (\exists X\subset \N)(\forall n\in \N)(n\in X\asa A_{0}(n)).
\]
\end{princ}
Finally, we mention some historical remarks about $\BOOT$.
\begin{rem}[Historical notes]\label{hist}\rm
First of all, $\BOOT$ is called the `bootstrap' principle as it is rather weak in isolation (equivalent to $\ACA_{0}$ under $\ECF$, in fact), but becomes much stronger 
when combined with comprehension axioms: $\SIXK+\BOOT$ readily proves $\Pi_{k+1}^{1}\textup{-}\textsf{CA}_{0}$.

\smallskip

Secondly, $\BOOT$ is definable in Hilbert-Bernays' system $H$ from the \emph{Grundlagen der Mathematik} (see \cite{hillebilly2}*{Supplement IV}).  In particular, one uses the functional $\nu$ from \cite{hillebilly2}*{p.\ 479} to define the set $X$ from $\BOOT$.  
In this way, $\BOOT$ and subsystems of second-order arithmetic can be said to `go back' to the \emph{Grundlagen} in equal measure, although such claims may be controversial.  

\smallskip

Thirdly, after the completion of \cite{samph}, it was observed by the author that Feferman's `projection' axiom \textsf{(Proj1)} from \cite{littlefef} is similar to $\BOOT$.  The former is however formulated using sets (and set parameters), which makes it more `explosive' than $\BOOT$ in that full $\Z_{2}$ follows when combined with $(\mu^{2})$, as noted in \cite{littlefef}*{I-12}.
\end{rem}

\section{Reverse Mathematics and countable sets}
\subsection{Introduction}
In this section, we obtain the results about countable sets sketched in Section~\ref{detail}, namely on limit point (and related notions of) compactness (Section \ref{LP}), K\"onig's lemma (Section~\ref{waycool}), and theorems from the RM zoo (Section~\ref{krazy}).  In doing so, we shall provide interesting answers to (Q0)-(Q3) from Section \ref{detail}.   
In Section \ref{introc}, we introduce some necessary definitions and obtain preliminary results; the latter provide motivation for the choice of the former. 

\smallskip

We single out the `explosive' theorems from Section \ref{LP}, i.e.\ the latter become much stronger when combined with discontinuous (comprehension) functionals.  
Our `previous best' was the Lindel\"of lemma for $\N^{\N}$ (see \cite{dagsamV}), which is weak in isolation but yields $\FIVE$ when combined with $(\exists^{2})$.  
In this paper, we identify a number of theorems that yield $\SIX$ when combined with $\FIVE^{\omega}$.
In this light, formalising rather basic mathematics in weak (third-order) systems seems difficult.  

\smallskip

Another important conceptual point is that results like Corollaries \ref{crucrucor} and \ref{surplus} 
establish that 
our results do not really depend on the exact notion of set used in this paper.  In fact, these corollaries show that the `logical power' (or lack thereof) of theorems about countable sets derives from the existence of injections (or bijections), and not from the particular notion of set used.  
Following the title of \cite{dsliceke}, our results can be said to be slicing the \emph{whole} truth, and nothing but the truth.  

\smallskip

Finally, we should mention the `excluded middle trick' pioneered in \cite{dagsamV}, as well as two related results by Kohlenbach from \cite{kohlenbach2, kohlenbach4} essential to this paper.
\begin{rem}[Two continuity results]\label{frorem}\rm
First of all, combining \cite{kohlenbach2}*{Prop.\ 3.7 and~3.12}, the following are equivalent to $(\exists^{2})$ over $\RCAo$:
\begin{itemize}
\item there is a function $f:\R\di \R$ that is not everywhere continuous, 
\item there is a function $g:\N^{\N}\di \N^{\N}$ that is not everywhere continuous,
\end{itemize}
where the usual `$\epsilon$-$\delta$-definition' of continuity is used.  
In particular, we note that $\neg(\exists^{2})$ is equivalent to `all functions on $\R$ (or $\N^{\N}$) are continuous', i.e.\ a version of Brouwer's (in)famous continuity theorem, over $\RCAo$.

\smallskip

Secondly, the standard `$\epsilon$-$\delta$-definition' of continuity on $\N^{\N}$ coincides with the second-order definition involving `RM-codes' from \cite{simpson2}*{II.6.1} in a weak system, as follows.   
Indeed, by combining \cite{kohlenbach4}*{Prop.\ 4.4 and 4.10}, $\RCAo+\WKL$ proves: 
\begin{center}
if $F:\N^{\N}\di \N$ is $\epsilon$-$\delta$-continuous on $2^{\N}$, then there is an RM-code $\alpha:\N\di \N$ that coincides with $F$, i.e.\ we have $(\forall f\in 2^{\N})(\exists N\in \N)(F(f)=\alpha(\overline{f}N))$.
\end{center}
The final formula is usually described as `$\alpha(f)$ is defined and equals $F(f)$'.  The same result holds \emph{mutatis mutandis} for functions on $\R$ that are continuous on $[0,1]$.   

\end{rem}  
\begin{rem}[The excluded middle trick]\label{LEMke}\rm
The law of excluded middle as in $(\exists^{2})\vee \neg(\exists^{2})$ is quite useful as follows:  suppose we are proving $T\di \NIN$ over $\RCAo$.  
Now, in case $\neg(\exists^{2})$, all functions on $\R$ are continuous by Remark \ref{frorem} and $\NIN$ holds trivially.  Hence, what remains is to establish $T\di \NIN$ 
\emph{in case we have} $(\exists^{2})$.  However, the latter axiom e.g.\ implies $\ACA_{0}$ and can uniformly convert reals to their binary representations.  
In this way, finding a proof in $\RCAo+(\exists^{2})$ is `much easier' than finding a proof in $\RCAo$.
In a nutshell, we may wlog assume $(\exists^{2})$ when proving theorems that are trivial (or readily proved) when all functions (on $\R$ or $\N^{\N }$) are continuous, like $\NIN$.   
A considerable `extension' of this trick is formulated in Remark \ref{flop}.
\end{rem}
\subsection{Countable sets in higher-order Reverse Mathematics}\label{introc}
We introduce our notion of `(strongly) countable set' in $\RCAo$, namely Definitions \ref{openset} and \ref{standard}, and formulate some associated principles, to be studied below.
We also establish preliminary results that justify our choice of the aforementioned definitions.  In particular, Corollary \ref{crucrucor} suggests that our results do not depend on the notion of set used, while Theorem \ref{crucru} shows that one cannot hope to study \emph{in systems below $\ACA_{0}$} third-order injections and bijections to $\N$ for sets in the sense of second-order RM, as defined in \cite{simpson2}*{II.5} or \eqref{morg} in Section \ref{defke} right below.

\subsubsection{Definitions}\label{defke}
In this paper, sets are given as in Definition \ref{openset}, namely via characteristic functions as in \cites{kruisje, dagsamVI, dagsamVII}. In fact, Definition \ref{openset} is taken from \cite{dagsamVII} and is inspired by the RM-definition of open and closed sets (\cite{simpson2}*{II.5.6}), as follows.

\smallskip

As is well-known, in second-order RM, closed sets are the complements of open sets and an open set $U\subset \R$
is represented by some sequence of open balls $\cup_{n\in \N}B(a_{n}, r_{n})$ where $a_{n}, r_{n}$ are rationals (see \cite{simpson2}*{I.4}).  One then writes the following for any $x\in \R$:
\be\label{morg}\tag{$\textsf{O}$}
x\in U \textup{ if and only if } {(\exists n\in \N )( |x-a_{n}|<_{\R} r_{n}  )}.  
\ee
Open and closed sets are well-studied in RM (see e.g.\ \cites{brownphd, browner, browner2}).
By \cite{simpson2}*{II.7.1}, one can effectively convert between RM-open sets and (RM-codes for) continuous characteristic functions.
Thus, a natural extension of \eqref{morg} is to allow \emph{arbitrary} (possibly discontinuous) characteristic functions, namely as follows. 
\bdefi[Sets in $\RCAo$]\label{openset}
We let $Y: \R \di \R$ represent subsets of $\R$ as follows: we write `$x \in Y$' for `$Y(x)>_{\R}0$' and call a set $Y\subseteq \R$ Ôopen' if for every $x \in Y$, there is an open ball $B(x, r) \subset Y$ with $r^{0}>0$.  
A set $Y$ is called `closed' if the complement, denoted $Y^{c}=\{x\in \R: x\not \in Y \}$, is open. 
\edefi
\noindent
Note that for open $Y$ as in the previous definition, the formula `$x\in Y$', which is `$Y(x)>_{\R}0$', has the same complexity (modulo higher types) as \eqref{morg}, while given $(\exists^{2})$ it becomes a `proper' characteristic function, only taking values `0' and `$1$'.  Hereafter, an `(open) set' refers to Definition \ref{openset}, while `RM-open set' refers to \eqref{morg}.

\smallskip

We shall use the following definition of `countable set' in $\RCAo$. 
\bdefi[Countable subset of $\R$]\label{standard}~
A set $A\subseteq \R$ is \emph{countable} if there exists $Y:\R\di \N$ such that $(\forall x, y\in A)(Y(x)=_{0}Y(y)\di x=_{\R}y)$. 
If $Y:\R\di \N$ is also \emph{surjective}, i.e.\ $(\forall n\in \N)(\exists x\in A)(Y(x)=n)$, we call $A$ \emph{strongly countable}.
\edefi
The first part of Definition \ref{standard} is from Kunen's set theory textbook (\cite{kunen}*{p.~63}) and the second part is taken from Hrbacek-Jech's set theory textbook \cite{hrbacekjech} (where the term `countable' is used instead of `strongly countable').  
For the rest of this paper, `strongly countable' and `countable' shall exclusively refer to Definition \ref{standard}, \emph{except when explicitly stated otherwise}. 

\smallskip

Finally, it behoves us to briefly motivate our choice of definitions, as follows.
\begin{itemize}
\item The RM-definition of closed set and third-order injections/bijections to $\N$ \emph{readily} give rise to $(\exists^{2})$ by Theorem \ref{crucru}, i.e.\ this combination does not allow us to work in systems below $\ACA_{0}$, suggesting the need for Definition \ref{openset}.
\item Some of our results do not really depend on the exact notion of set used by e.g.\ Corollaries \ref{crucrucor} and \ref{surplus}.
In particular, the `logical power' (or lack thereof) of theorems about countable sets derives from the existence of injections (or bijections) to $\N$, and not from the particular notion of set.  
\item Definition \ref{openset} allows us to `split' a given theorem, like item~\eqref{ob123} in Theorem~\ref{angels}, into a well-known second-order part and a third-order `explosive' part without first-order strength, namely as in item \eqref{ob12} in Theorem \ref{angels}.  
\end{itemize}
These results are obtained in the next sections as corollaries to our main results.  

\subsubsection{Basic properties of countable sets}\label{florhi}
It seems obvious that a sequence in $\R$ forms a countable set, but we show in this section that proving this fact requires non-trivial axioms.  The same holds for seemingly trivial statements from topology, namely related to countability axioms and $\N_{\infty}$, which we introduce next.   These results motivate our choice of definitions made in the previous section.   In particular, by Corollary \ref{crucrucor}, our results do not really depend on 
the exact definition of set.  

\smallskip

First of all, the set $\N_{\infty}$ is the \emph{one-point compactification} of $\N$, defined as follows:
\be\label{nek}
\N_{\infty}:=\{ f\in 2^{\N}:  \underline{(\forall n\in \N)(f(n)=1 \di (\forall m<n)(f(m)=1)}   \}.
\ee
We interpret `$f\in \N_{\infty}$' as the underlined $\Pi_{1}^{0}$-formula from \eqref{nek}, i.e.\ $\N_{\infty}$ is an RM-closed set and items \eqref{klonk2} and \eqref{klonk} in Theorem \ref{crucru} express that $\N_{\infty}$ is (strongly) countable.  Thus, Theorem \ref{crucru} shows that one cannot hope to study bijections or injections on second-order RM-sets in systems below $\ACA_{0}$, which provides a modicum of motivation for our definitions in the previous section.  
Our study of $\N_{\infty}$ as in Theorem \ref{crucru} was inspired by the decidability results pertaining to the latter in constructive mathematics from \cite{escardokes}.
 
\smallskip

Secondly, we note that Munkres defines \emph{first-countable} both in terms of countable sets (\cite{munkies}*{p.\ 190}) and sequences (\cite{munkies}*{p.\ 131}).
By contrast, the notion \emph{second-countable} is only defined based on countable sets by Munkres in \cite{munkies}. 
We use the aforementioned standard definitions\footnote{A topological space $X$ is \emph{first-countable} if for any $x \in X$ there is a countable collection $\mathfrak{B}$ of neighbourhoods of $x$ such that each neighbourbood of $x$ contains some $B\in \mathfrak{B}$ (\cite{munkies}*{p.\ 190}).  A topological space $X$ is \emph{second-countable} if there is a countable collection $\mathfrak{U}$ of basic opens sets such that any open subset of can be written as a union of elements of $\mathfrak{U}$ (\cite{munkies}*{p.\ 78}.)} based on countable sets.   Similar to the previous paragraph, Theorem \ref{crucru} provides motivation for our definitions from Section~\ref{defke} and   
a partial answer to (Q0).  
\begin{thm}\label{crucru}
The following are equivalent over $\RCAo$\textup{:}
\begin{enumerate}
\renewcommand{\theenumi}{\roman{enumi}}
\item $(\exists^{2})$,\label{discon}
\item any sequence $(x_{n})_{n\in \N}$ of reals is a countable set,  \label{watte}
\item the unit interval is first-countable, \label{wataru}
\item  $(\exists Y^{2})(\forall f, g \in \N_{\infty})(Y(f)=_{0}Y(g)\di f=_{1}g) $,\label{klonk2}
\item item \eqref{klonk2} where $Y^{2}$ additionally satisfies $ (\forall n^{0})(\exists h\in \N_{\infty})(Y(h)=n)$.\label{klonk}
\end{enumerate}
Over $\RCAo$, $(\exists^{2})$ follows from the statement: \textup{the unit interval is second-countable.}
\end{thm}
\begin{proof}
For the implication $\eqref{discon}\di \eqref{klonk}$, note that $f\in \N_{\infty}$ means that either $f=_{1}00\dots$, $f=_{1}11\dots$, or $f=_{1}\sigma_{k}*00\dots$ for $k\geq 1$, where $\sigma_{k}(i):=1$ for $i<|\sigma_{k}|=k\geq 0$.  
Using $\exists^{2}$ to distinguish between these cases, define $Y(f)$ as respectively $0$, $1$, $k+1$ for $|\sigma|=k>0$.  Clearly, $Y$ satisfies the properties from item \eqref{klonk} as $Y(\sigma_{k}*00\dots)=k+1$ for $k\geq 1$ and $Y(00\dots)=0$ and $Y(11\dots)=1$.

\smallskip

For $\eqref{klonk2}\di\eqref{klonk} \di  \eqref{discon}$, we use the following equivalence from \cite{escardokes}:
\be\label{escard}
(\forall f\in \N_{\infty})(Y(f)>0)\asa Y(\eps(Y))>0,
\ee
where $\eps^{2\di 1}$ is defined as follows: $\eps(Y)(n)$ is $1$ if $Y(\sigma_{k}*00\dots)>0$ for all $k\leq n$, and zero otherwise.
The forward implication in \eqref{escard} is trivial, while the reverse implication follows from the aforementioned case distinction. 
Indeed, assume $Y(\eps(Y))>0$ and note that if $Y(00\dots)=0$, then $\eps(Y)=_{1}00\dots$ by definition, a contradiction.  
Similarly, again assuming $Y(\eps(Y))>0$, if $Y(\sigma_{k}*00\dots)=0$ for $k\geq 1$ the least such number, then $\eps(Y)=_{1}\sigma_{k-1}*00\dots$, which is impossible.  
Indeed, $\eps(Y)=_{1}\sigma_{k-1}*00\dots$ implies $0<Y(\eps(Y))=_{1}Y(\sigma_{k-1}*00\dots)$ by the extensionality of $Y$ and our assumption.  
However, $\eps(Y)=_{1}\sigma_{k-1}*00\dots$ also implies $\eps(Y)(k-1)=0$ and hence $Y(\sigma_{k-1}*00\dots)=0$ by definition.   
Hence, we have that \emph{if} $Y(\eps(Y))>0$, \emph{then} $Y( \sigma_{k}*00\dots)>0$ for any $k\in \N$.  
The consequent of the latter observation also yields that $\eps(Y)=_{1}11\dots$ by the definition of $\eps$.  But then also $Y(11\dots)=Y(\eps(Y))>0$ by the extensionality of $Y$ and our assumption.  
In this way, we have established the left-hand side of \eqref{escard}, which tells us that quantifiers over $\N_{\infty}$ are \emph{decidable}; this is proved in \cite{escardokes} for constructive mathematics, which is all the more impressive.  Now let $Y^{2}$ be as in item \eqref{klonk} and consider:  
\be\label{hyuk}
(\exists n^{0})(f(n)=0)\asa (\exists g\in \N_{\infty})(f(Y(g))=0)\asa f(Y(\eps(\lambda g.f(Y(g)))))=0,
\ee
which immediately provides us with $(\exists^{2})$.  To obtain $\eqref{klonk2}\di\eqref{klonk}$, let $Y^{2}$ be as in the former and consider the following formula:
\be\label{flap}
(\forall f \in \N_{\infty})(Y(f)\ne_{0} Y(11\dots) \di (\exists k\in \N)(f(k)=0)  ),
\ee
which has the right format for applying $\QFAC^{1,0}$ (included in $\RCAo$) as `$f\in \N_{\infty}$' is $\Pi_{1}^{0}$.
The `trick' we use is that the $\Sigma_{1}^{0}$-formula $f\ne_{1}11\dots $ can be replaced by the \emph{decidable} formula $Y(f)\ne_{0}Y(11\dots)$, thanks to our assumptions on $Y$.
Thus, let $G^{2}$ be such that $G(f)=k$ is the least $k\in \N$ such that $f(k)=0$ for $f\in \N_{\infty}$ such that $Y(f)\ne Y(11\dots)$.
Now define $Y_{0}^{2}$ as follows:
\[
Y_{0}(f):=
\begin{cases}
0 & Y(f)=Y(00\dots)\\
1 & Y(f)=Y(11\dots)\\
G(f)+1 & \textup{ otherwise }
\end{cases},
\]
which has the right properties for item \eqref{klonk} by the definition of $G$.
Note that we could also use $(\forall n\in \N)(f(n)>0)\asa Y(\tilde{f})=Y(11\dots)$ to obtain $(\exists^{2})$, where $\tilde{f}(n)=1$ if $(\forall i\leq  n)(f(i)=1)$, and zero otherwise.   Note $\tilde{f}\in \N_{\infty}$ for all $f\in 2^{\N}$.

\smallskip

For the implication \eqref{discon}$\di$\eqref{watte}, let $(x_{n})_{n\in \N}$ be a sequence of reals and define $x\in A$ as $(\exists n\in \N)(x=_{\R}x_{n})$ using $\exists^{2}$.  
Use $\mu^{2}$ to define $Y:\R\di \N$ as follows: $Y(x)$ is the least $n\in \N$ such that $x=_{\R}x_{n-1}$ if such exists, and $0$ otherwise.  
Then $Y$ is an injection on $A$ by definition, and the latter set is therefore countable. 

\smallskip

Next, assume \eqref{watte} and suppose $\neg(\exists^{2})$.  The latter implies that all functions on $\R$ are continuous by Remark \ref{frorem}. 
Now fix some countable set $A\subset \R$, i.e.\ there is $Y:\R\di \N$ that is injective on $A$.  By continuity, if $x_{0}\in A$, then for $y\in \R$ close enough to $x_{0}$, we have $y\in A$ by Definition \ref{openset}. 
Again by continuity, if $Y(x_{0})=n_{0}$ for $x_{0}\in A$, then for $z\in \R$ close enough to $x_{0}\in A$, we have $Y(z)=n_{0}$, a contradiction.  
Hence, $A$ must be empty and the same holds for all countable sets.  However, item \eqref{watte} applied to e.g.\ $(\pi+\frac{1}{2^{n}})_{n\in \N}$ gives rise to a non-empty set, a contradiction.  Hence, we must have $(\exists^{2})$, and $\eqref{discon}\asa \eqref{watte}$ thus follows.

\smallskip

As to $\eqref{discon}\di\eqref{wataru}$, to show that $[0,1]$ is first-countable, one uses $\exists^{2}$ to define for $x\in [0,1]$ the countable set $A_{x}$ consisting of all intervals in $[0,1]$ with rational end points and containing $x$.
For $\eqref{wataru}\di \eqref{discon}$, in the same way as in the previous paragraph, $\neg(\exists^{2})$ leads to a contradiction as the local basis as in the definition of first-countability cannot be empty. 
The final part follows in the same way as in the previous two paragraphs as a basis for a topology cannot be empty. 
\end{proof}
The previous theorem is important as follows: the topology of first-countable spaces can be described in terms of sequences, while other spaces 
require nets for this purpose.  When $(\exists^{2})$ is not available, we are therefore more of less forced to use nets for describing the topology of $[0,1]$, assuming coding is not an option.  The (higher-order) RM-study of nets can be found in \cites{samnetspilot, samcie19, samwollic19}.
Moreover, modulo technical details, the previous theorem suggests we can derive $(\exists^{2})$ from the existence of a countable set in many contexts, 
like the existence of a countable basis for vector spaces; the associated (second-order) RM study is in \cite{simpson2}*{III.4}.  

\smallskip

Next, the equivalence $\eqref{discon}\asa\eqref{watte}$ in Theorem \ref{crucru} is \emph{partly} due to our choice of the notion of set, namely as in Definition \ref{openset}. 
\emph{Nonetheless}, the next corollary shows that $\eqref{watte}\di \eqref{discon}$ goes through over $\RCAo+\WKL$ \emph{no matter what notion of set is used}.  
This `meta-result' is not fully formal, but rigorous all the same. 
\begin{cor}\label{crucrucor}
The implication $\eqref{watte}\di \eqref{discon}$ from the theorem goes through over $\RCAo+\WKL$ regardless of the meaning of `$x\in A$'.
\end{cor}
\begin{proof}
Assume \eqref{watte} from the theorem and suppose $\neg(\exists^{2})$.  The latter implies that all functions on $\R$ are continuous by Remark \ref{frorem}.
Now fix some set $A\subset [0,1]$ (whatever it may be) that is countable, i.e.\ there is $Y:\R\di \N$ that is injective on $A$.  Since $Y$ is continuous on $[0,1]$, it has an RM-code given $\WKL$ following Remark~\ref{frorem}.
By the RM of $\WKL$ in \cite{simpson2}*{IV}, this RM-code, and hence $Y$, is also \emph{uniformly} continuous on $[0,1]$.  
Let $N_{0}\in \N$ be such that $|Y(x)-Y(y)|<1$ for any $x, y\in [0,1]$ with $|x-y|<\frac{1}{2^{N_{0}}}$ and note that $Y$ is must be constant on $[0,1]$.
Hence, whatever `$x\in A$' is, it can only be true for at most one $x$ in $[0,1]$, since $Y$ is an injection.  
However, item \eqref{watte} guarantees that the sequence $(\frac{\pi}{2^{n+2}})_{n\in \N}$ in $[0,1]$ gives rise to a countable set with infinitely many distinct elements, a contradiction. 
\end{proof}
The reader may verify that the below results satisfy `independence' properties similar to Corollary \ref{crucrucor}.  We have obtained such results in e.g.\ Corollary \ref{surplus}.
We discuss the implications of these corollaries in Remark \ref{flop}.

\subsubsection{Basic principles about \(un\)countable sets}\label{bumm}
We introduce some fundamental principles to be studied below.

\smallskip

First of all, the following\footnote{It is a tedious-but-straightforward verification that the below proofs still go through if we replace the equivalence by a forward arrow in $\cocode_{0}$.  One `immediate' example is that $\cocode_{0}\di \BOOT_{C}^{-}$ in the proof of Theorem \ref{angel}.} 
coding principle is central to \cite{dagsamX, dagsamXI} and this paper. 
\begin{princ}[$\cocode_{0}$]
For any non-empty countable set $A\subseteq [0,1]$, there is a sequence $(x_{n})_{n\in \N}$ in $A$ such that $(\forall x\in \R)(x\in A\asa (\exists n\in \N)(x_{n}=_{\R}x))$.
\end{princ}
We let $\fin_{0}$ be $\cocode_{0}$ restricted as follows.  
\begin{princ}[$\fin_{0}$]
Let $A\subset[0,1]$ and $Y:[0,1]\di \N$ be such that the latter is injective \emph{and bounded} on the former.  
Then there is a sequence $(x_{n})_{n\in \N}$ in $A$ such that $(\forall x\in \R)(x\in A\asa (\exists n\in \N)(x_{n}=_{\R}x))$.
\end{princ}
Intuitively speaking, $\fin_{0}$ expresses that finite sets can be listed by sequences.
Moreover, we let $\cocode_{1}$ be $\cocode_{0}$ restricted to \emph{strongly} countable sets. 
As will become clear in Section \ref{LP}, there is a big difference between theorems formulated with `countable' and with `strongly countable' in that the former (but not the latter) are \emph{explosive}, i.e.\ become much 
stronger when combined with a discontinuous (comprehension) functionals, even up to $\SIX$, the current upper bound of RM, to the best of our knowledge.

\smallskip

Secondly,  injections and bijections are connected via the \emph{Cantor-Bernstein} theorem, formulated by Cantor in \cite{cantor3} and first proved by Bernstein (see \cite{opborrelen2}*{p.\ 104}).
Intuitively, if there is an injection from $A$ to $B$ and an injection from $B$ to $A$, then there is a bijection between them. 
We study the restriction to $B=\N$ and $A\subset \R$.  
Note that item \eqref{wr} in $\CBN$ provides the second injection (from $\N$ to $B$) in the usual formulation of the Cantor-Bernstein theorem.  
\begin{princ}[$\CBN$]
For a set $A\subset \R$, the following two conditions:
\begin{enumerate}
\renewcommand{\theenumi}{\roman{enumi}}
\item there is $Y:\R\di \N$ which is injective on $A$, i.e.\ $A$ is \textbf{countable}
\item there is a sequence $(x_{n})_{n\in \N}$ in $A$ consisting of pairwise distinct reals,\label{wr}
\end{enumerate}
imply that there is $Z:\R\di \N$ which is bijective on $A$, i.e.\ $A$ is \textbf{strongly countable}.  
\end{princ}
\noindent
By Theorem \ref{angel}, $\CBN$ and $\cocode_{0}$ are intimately connected, and both are `explosive' by Corollary \ref{BOEM}.
By contrast, $\N_{\infty}$ `trivially' satisfies the Cantor-Bernstein theorem by (the proof of) Theorem \ref{crucru}, which provides another argument against using RM-closed sets rather than Definition \ref{openset}.
As it happens, our study of $\N_{\infty}$ was motivated by the study of the Cantor-Bernstein theorem in constructive (reverse) mathematics via $\N_{\infty}$ in \cite{cblem}. 

\smallskip

Now, the usual formulation of the Cantor-Bernstein theorem involves two injections.  Let $\CBN'$ be $\CBN$ with the `sequence' condition replaced by `there is an injection from $\N$ to $A$', which is perhaps a rather naive formulation of the Cantor-Bernstein theorem.  
The following principle is the main subject of \cite{dagsamIX} and studied in \cite{dagsamX}, and connects these two versions. 
\begin{princ}[$\NCC$]\label{frik}
For $Y^{2}$ and $A(n, m)\equiv (\exists f\in 2^{\N})(Y(f, m, n)=0)$:
\be\label{garf}
(\forall n \in \N)(\exists m \in \N)A(n,m)\di  (\exists g:\N\di \N)(\forall n\in \N)A(n,g(n)). 
\ee
\end{princ}
Intuitively speaking, $\NCC$ is a very weak version of countable choice as in $\QFAC^{0,1}$, with interesting RM and computability properties, as studied in \cite{dagsamIX}.

\smallskip

Finally, while $\R$ and $[0,1]$ are `obviously' not countable, the following principles cannot be proved in $\Z_{2}^{\omega}$ (but $\Z_{2}^{\Omega}$ can prove them), as shown in \cite{dagsamX}*{\S3}.
\begin{princ}[$\NIN$]
For $Y:[0,1]\di \N$, there are $x, y\in [0,1]$ such that $x\ne_{\R} y$ and $Y(x)=_{\N}Y(y)$. 
\end{princ}
\begin{princ}[$\NBI$]
For $Y:[0,1]\di \N$, \textbf{either} there are $x, y\in [0,1]$ such that $x\ne_{\R} y$ and $Y(x)=_{\N}Y(y)$, \textbf{or} there is $N\in \N$ such that $(\forall x\in [0,1])(Y(x)\ne N)$.
\end{princ}
As is clear from \cite{dagsamX}*{Figure 1}, $\NIN$ and $\NBI$ follow from many elementary theorems of ordinary mathematics. 
As it turns out, deriving $\NIN$ (or $\NBI$) often also allows us to derive $\cocode_{0}$ (or $\cocode_{1}$), while the former proofs are conceptually simpler, in general. 
We shall provide examples of this phenomenon in this paper.   

\smallskip

The following theorem connects $\NIN$ with \eqref{hong} and \eqref{hong2} from Section \ref{bintro}.   
Note that the implication $\eqref{hong2}\di \NIN$ can be found in e.g.\ Rudin's textbook \cite{rudin}*{p.\ 41}.
\begin{thm}\label{momo}
The system $\RCAo$ proves $\NIN\asa \eqref{hong}$ and $\eqref{hong2}\di \NIN$.
The same holds for $\NBI$ and strongly countable sets. 
\end{thm}
\begin{proof}
For the equivalence, $\NIN$ and \eqref{hong} are trivial in case $\neg(\exists^{2})$, by the proof of Theorem \ref{crucru}.
In case $(\exists^{2})$, $[0,1]$ is a set in the sense of Definition \ref{openset}.  
The equivalence is now a mere manipulation of definitions using contraposition. 
The equivalence involving $\NBI$ and \eqref{hong} restricted to strongly countable sets is proved in the same way. 
For $\eqref{hong2}\di \NIN$, 
 $\neg\NIN$ implies that $[0,1]$ is countable and perfect.  
\end{proof}
Note that $\eqref{hong2}\di \NIN$ goes through for any definition of `perfect set' that makes the unit interval a perfect set in $\RCAo$.     
Since $\NIN$ does not mention the notion of `countable set', the equivalence $\NIN\asa \eqref{hong}$ suggests that our notion of set as in Definition \ref{openset} is relatively satisfactory.  

\smallskip

Finally, we discuss the implications of Corollary \ref{crucrucor} for principles concerning countable sets like $\cocode_{0}$. 
\begin{rem}[Countable sets and continuity]\label{flop}\rm
The aim of Section \ref{florhi} was to show that the exact definition of set does not matter much when dealing with countable sets, culminating in Corollary \ref{crucrucor}. 
In this light, we might as well chose the `most convenient' definition of (countable) set, namely Definition \ref{openset}.  Indeed, given $\neg(\exists^{2})$, one readily\footnote{Let $Y:[0,1]\di \R$ be injective on $A\subset[0,1]$ where the latter is represented by $Z:[0,1]\di \N$, i.e.\ $x\in A\asa Z(x)=1$ for $x\in \R$.  Given $\neg(\exists^{2})$, $Z, Y$ are continuous by the above.  Hence, for $x\in A$, we have $y\in A$ for $y\in B(x, \frac{1}{2^{N}})$, for some $N\in \N$.  Moreover, for some $M\in \N$, $Y(z)=Y(x)$ for all $z\in B(x, \frac{1}{2^{M}})$.  Considering $\max(N,M)$, we observe that $Y$ is not injective on $A$.\label{score}} shows that 
\begin{center}
\emph{there are no non-empty countable subsets of $\R$ or $\N^{\N}$}, 
\end{center}
rendering the principles $\cocode_{0}$ and $\NIN$ trivial.  Below, we shall often make use of this observation, which we have delayed until now for didactic reasons. 
The same observation holds for alternative representations of sets, which we will explore in a future publication. 
\end{rem}

\subsection{Limit point versus sequential compactness}\label{LP}
The main topic of this section is Section \ref{dir1}, namely the study of compactness via nets with countable index sets.  
These results give rise to improvements of results on nets with \emph{uncountable index sets} from e.g.\ \cite{samph, samnetspilot}, as explored in Section \ref{dir2}
\subsubsection{Nets with countable index sets}\label{dir1}
In this section, we study limit point compactness and nets \emph{restricted to countable index sets}, with an eye on the questions (Q0)-(Q3) from Section \ref{detail}.

\smallskip

Now, the RM-study of \emph{sequential} compactness is well-known (\cite{simpson2}*{III.2}); the RM-study of \emph{limit-point} compactness\footnote{A space is called \emph{limit point compact} by Munkres in \cite{munkies}*{p.\ 178} if every infinite subset has a limit (=adherence) point.} and compactness \emph{based on nets} is not as well-developed.  In the latter case,  \cite{samnetspilot} constitutes a first step and may be consulted for an introduction to nets in $\RCAo$, including the required definitions, which are of course essentially the standard ones from the literature.  

\smallskip

First of all, let $\BW_{0}^{[0,1]}$ be the following version of the Bolzano-Weierstrass theorem.
We discuss the details of the formulation of this principle in Remark \ref{dagki}.
\begin{princ}[$\BW_{0}^{[0,1]}$]\label{konk}
For countable $A\subset [0,1]$ and $F:[0,1]\di [0,1]$, the supremum $\sup_{x\in A}F(x)$ exists. 
\end{princ}
Using the usual interval-halving technique and working over $\RCAo$, $\ACA_{0}$ is equivalent to the statement that for a \emph{sequence} $(x_{n})_{n\in \N}$ and \textbf{any} $F:[0,1]\di [0,1] $, $\sup_{n\in \N}F(x_{n})$ exists.  
As is clear from Theorem \ref{angel} and its corollaries, the formulation using countable sets as in $\BW_{0}^{[0,1]}$ behaves \emph{quite} differently. 

\smallskip

We have studied the RM of nets in \cite{samnetspilot}, and we refer to the latter for all relevant definitions related to nets in $\RCAo$.  
While nets can have almost arbitrary index sets, Tukey shows in \cite{tukey1} that topology 
can be developed based on \emph{phalanxes}, which are nets with index set consisting of the finite subsets of a given set, ordered by inclusion. 
Thus, let $\MCT_{0}^{\net}$ be the monotone convergence theorem for nets in $[0,1]$ with index set consisting of finite sequences without repetitions, as follows:
\be\label{basik}
D=\{ w^{1^{*}}: (\forall i, j <|w|)(i\ne j\di w(i)\ne w(j))\wedge (\forall i<|w|)(w(i)\in A)  \},
\ee
for countable $A\subset \R$.  We order $D$ by inclusion, i.e.\ `$w\preceq_{D}v$' means that all elements of $w$ are included in $v$. 
We note that e.g.\ \cite{samnetspilot}*{\S4} deals with nets that are not obviously phalanxes, in the context of sequential continuity and compactness.   

\smallskip

Now, the monotone convergence theorem for sequences is equivalent to $\ACA_{0}$ (\cite{simpson2}*{III.2}), i.e.\ the formulation using countable sets 
as in $\MCT_{0}^{\net}$ behaves \emph{quite} differently by the following theorem.
\begin{thm}\label{angel}
The system $\ACAo$ proves that the following are equivalent:
\begin{enumerate}
\renewcommand{\theenumi}{\roman{enumi}}
\item $\BW_{0}^{[0,1]}$: Bolzano-Weierstrass theorem for countable sets in $[0,1]$,\label{ob1}
\item $\MCT_{0}^{\net}$: monotone convergence theorem for nets with countable index sets,\label{ob2}
\item $\BOOT^{-}_{C}$,\label{ob3}
\item $ \cocode_{0}$: a countable set in $[0,1]$ can be listed as a sequence, \label{ob4}
\item $\CBN+\cocode_{1}$.\label{ob5}
%
%

\end{enumerate}
Omitting item \eqref{ob2}, the equivalences go through over $\RCAo$.
\end{thm}
\begin{proof}
First of all, while $(\exists^{2})$ is needed to state \eqref{basik} and hence item \eqref{ob2}, all other items are vacuously true given $\neg(\exists^{2})$, following the penultimate paragraph of the proof of Theorem \ref{crucru}, and Remark \ref{flop}. 
Hence, the final sentence of the theorem involving $\RCAo$ follows from equivalences over $\ACAo$ by considering the law of excluded middle $(\exists^{2})\vee \neg(\exists^{2})$.  
For the rest of the proof, we assume $(\exists^{2})$ and hence $\ACA_{0}$ and attendant RM-results from \cite{simpson2}*{III.2}.
In particular, we may use $\exists^{2}$ to freely convert between reals in $[0,1]$ and their binary representations.  We shall therefore sometimes work over $2^{\N}$ rather than $[0,1]$.
Moreover, elementhood (as in Definition \ref{openset}) is decidable.    

\smallskip

Secondly, $\eqref{ob4}\di \eqref{ob1}$ is immediate as $\cocode_{0}$ converts a countable set in $[0,1]$ into a sequence, which
has a supremum given $\ACA_{0}$ by \cite{simpson2}*{III.2}.  Similarly, $\eqref{ob4}\di \eqref{ob2}$ follows from the observation 
that $\cocode_{0}$ converts a countable index set into a sequence, after which the usual `interval-halving technique' can be done using $\ACA_{0}$.

\smallskip

The implication $\BOOT^{-}\di \cocode_{0}$ is proved in \cite{dagsamX}*{Theorem 3.16}; since $\cocode_{0}$ deals with sets in $[0,1]$, the same proof also yields $\eqref{ob3}\di \eqref{ob4}$ by coding real numbers as binary sequences using $\exists^{2}$.  
For completeness, we provide a sketch as follows.  
Fix $A\subset [0,1]$ and let $Y:[0,1]\di \N$ be an injection on $A$.  
Use $\BOOT^{-}_{C}$ to define $X\subset \N^{2}\times \Q$ such that for all $n, m\in \N$ and $q\in \Q\cap [0,1]$, we have:
\be\label{kilop}\textstyle
(n, m, q)\in X\asa (\exists x\in B(q, \frac{1}{2^{m}})\cap A)(Y(x)= n ).
\ee
By assumption, the following condition required for $\BOOT^{-}_{C}$, is satisfied: 
\[\textstyle
(\forall n, m\in \N, q\in \Q\cap [0,1])(\exists \textup{ at most one } x\in B(q, \frac{1}{2^{m}})\cap A)(Y(x)= n ).
\]
We now use $X$ from \eqref{kilop} and the well-known interval-halving technique to create a sequence $(x_{n})_{n\in \N}$.  
For fixed $n\in \N$, define $[x_{n}](0)$ as $0$ if $(\exists x\in [0, 1/2)\cap A)(Y(x)=n)$, and $1/2$ otherwise; define $[x_{n}](m+1)$ as $[x_{n}](m)$ if $(\exists x\in \big[[x_{n}](m), [x_{n}](m)+\frac{1}{2^{m+1}}\big)\cap A\big)(Y(x)=n)$, and $[x_{n}](m)+\frac{1}{2^{m+1}}$ otherwise.  
By definition, we have: 
\[
(\forall n\in \N)\big[(\exists x\in A)(Y(x)=n)\asa [x_{n}\in A\wedge  Y(x_{n})=n] \big],
\]
which readily yields $\cocode_{0}$ as $\mu^{2}$ can be used to removed $x_{n}$ from the sequence in case $x_{n}\not\in A$.
This concludes the sketch based on the proof of \cite{dagsamX}*{Theorem 3.16}.

\smallskip

Next, we prove $\eqref{ob4}\di \eqref{ob3}$. 
Fix $Z$ such that $(\forall n\in \N)(\exists \textup{ at most one } g\in 2^{\N})(Z(g, n)=0)$ and define the set $A:=\{ g\in 2^{\N}: (\exists n\in \N)(Z(g, n)=0))\}$. 
To show that $A$ is countable, define $Y(g)$ as the least $n\in \N$ such that $Z(g, n)=0$, if such there is, and zero otherwise; by assumption, $Y$ is an injection on $A$ and the latter is therefore countable. 
Let $(h_{m})_{m\in \N}$ be a sequence in $2^{\N}$ such that $g\in A\asa (\exists m\in \N)(g=_{1}h_{m})$ for any $g\in 2^{\N}$.  In this way, we have $(\exists g\in 2^{\N})(Z(g,n )=0)\asa (\exists m\in \N)(Z(h_{m}, n)=0) $ for any $n\in \N$, and $\BOOT^{-}_{C}$ now follows from $\ACA_{0}$. 

\smallskip

We now prove $\eqref{ob2}\di \eqref{ob3}$.  
Recall that $(\exists^{2})$ allows us to freely convert between real numbers and their binary representations, which means we can work over $2^{\N}$ rather than $[0,1]$ without problems.  
Fix $Z$ such that $(\forall n\in \N)(\exists \textup{ at most one } g\in 2^{\N})(Z(g, n)=0)$.
Let $A\subset 2^{\N}$ be the countable set as in the previous paragraph and let $Y:2^{\N}\di \N$ be the associated injection on $A$.
Let $D$ be the set of finite sequences in $A$ (without repetitions) and let $\preceq_{D}$ be the inclusion ordering, i.e.\ $w\preceq_{D}v$ if $(\forall i<|w|)(\exists j<|v|)(w(i)=_{1}v(j))$.  
Now define the net $f_{w}:D\di 2^{\N} $ as $f_{w}:=\lambda k.F(w, k)$ where $F(w, k)$ is $1$ if $(\exists i<|w|)(Y(w(i))=k)$, and zero otherwise. 
This net is increasing as $w\preceq_{D} v\di f_{w}\leq_{\lex} f_{v}$ by definition.  Let $h$ be the limit of this net and consider the following equivalences, for any $n\in \N$:
\be\label{cent}
 h(n)=1\asa (\exists f\in A)(Y(f)=n)\asa(\exists g\in 2^{\N})(Z(g,n)=0), 
\ee
The second equivalence in \eqref{cent} follows by the definition of $Y$ in terms of $Z$.  The first equivalence in \eqref{cent} follows from the 
definition of limit of a net.   By \eqref{cent}, the set $\{n \in \N: (\exists g\in 2^{\N})(Z(g,n)=0)\}$ exists, as required for $\BOOT_{C}^{-}$.

\smallskip

Next, $\eqref{ob1}\di \eqref{ob3}$ is obtained by modifying the previous paragraph as follows: the set $B=\{w^{1^{*}}: (\forall i<|w|)(w(i)\in A) \}$ is countable as $\r(f_{w})=_{\Q}\r(f_{v})\di f_{w}=_{1}f_{v}$, for $v^{1^{*}}, w^{1^{*}}$ finite sequences in $A$.
Modulo coding using $\exists^{2}$, $B$ can be viewed as a subset of $2^{\N}$.
The supremum of $\sup_{w\in B}F(w)$ for $F(w):=f_{w}$ also yields \eqref{cent} and hence $\BOOT^{-}_{C}$.  
Note that $\eqref{ob1}\di \eqref{ob3}$ is also proved in the proof of \cite{dagsamX}*{Theorem~3.23} for $\BW_{0}^{[0,1]}$ formulated using $2^{\N}$ rather than $[0,1]$.

\smallskip

Finally, the implications $\eqref{ob5}\di \eqref{ob4}\di \cocode_{1}$ are clearly trivial.  
The remaining implication $\cocode_{0}\di\CBN$ readily follows by using $\exists^{2}$ to `trim' duplicate reals from the sequence provided by $\cocode_{0}.$
\end{proof}
For the next corollary, note that $\cocode_{1}$ is weak {and} \emph{not} explosive\footnote{We have that $\QFAC^{0,1}\di \cocode_{1}$ over $\RCAo$ by \cite{dagsamX}*{Theorem 3.24} and that $\FIVE^{\omega}+\QFAC^{0,1}$ is a $\Pi_{3}^{1}$-conservative extension of $\FIVE$ by \cite{yamayamaharehare}*{Theorem 2.2}.\label{kul}}, i.e.\ the principle $\CBN$ is (mostly) responsible for obtaining $\SIX$.
\begin{cor}\label{BOEM}
Let $\X$ be any item among \eqref{ob1}-\eqref{ob5} from the theorem. The system $\FIVE^{\omega}+\X$ proves $\SIX$.
\end{cor}
\begin{proof}
Since $(\SS^{2})\di (\exists^{2})$, we may use the latter to freely convert between reals in $[0,1]$ and their binary representations.  
Hence, $\BW_{0}^{[0,1]}$ is readily seen to be equivalent to $\BW_{0}^{C}$, where the latter is the former but formulated for Cantor space $2^{\N}$, namely as in \cite{dagsamX}*{Def.\ 3.21}.  
By \cite{dagsamX}*{Theorem 3.22}, $\SIX$ follows from $\FIVE^{\omega}+\BW_{0}^{C}$, and we are done.  
\end{proof}
The following corollary establishes that $\BW_{0}^{[0,1]}$ is hard to prove, while it by itself does not yield $\L_{2}$-sentences beyond $\RCA_{0}$.  Corollary \ref{surplus} suggests that this kind of result does not depend on our choice of the notion of set used. 
Let $\X\in \L_{2}$ be any sentence not provable in $\RCA_{0}$.  
\begin{cor}\label{Schor}
The system $\Z_{2}^{\omega}+\QFAC^{0,1}$ cannot prove $\BW_{0}^{[0,1]}$. The system $\RCAo+\BW_{0}^{[0,1]}$ cannot prove $\X$.
\end{cor}
\begin{proof}
For the first part, by \cite{dagsamX}*{\S3}, $\Z_{2}^{\omega}+\QFAC^{0,1}$ cannot prove $\NIN$, while $\cocode_{0}\di \NIN$.  
For the remaining part, let $\X\in \L_{2}$ be a sentence not provable in $\RCA_{0}$.  Then $\RCAo+\BW_{0}\vdash \X$ yields $\RCA_{0}\vdash \X$ under $\ECF$ as the latter makes $\BW_{0}$ vacuously true, following the penultimate part of the proof of Theorem \ref{crucru}.  
\end{proof}
We could obtain similar results for the Ascoli-Arzel\`a theorem (see \cite{simpson2}*{III.2} for the associated RM results) for nets with countable index sets.  
The following corollary should be contrasted with \cite{samph}*{\S3.2}, where it is shown that the existence of a modulus of convergence for 
nets index by $\N^{\N}$, yields $\QFAC^{0,1}$.
\begin{cor}\label{Schor2}
The theorem remains valid if we require a modulus of convergence in the conclusion of $\MCT_{0}^{\net}$ or a convergent sub-sequence in $\BW_{0}^{[0,1]}$.
\end{cor}
\begin{proof}
Given $\cocode_{0}$, a countable index set is given by a sequence.   Hence, a modulus of convergence can be defined in the usual (arithmetical) way.
The same argument works for $\BW_{0}^{[0,1]}$.
\end{proof}
Obviously, we would like to obtain equivalences like in Theorem \ref{angel} for \emph{strongly} countable sets.  
Now, if we try to imitate the proof of e.g.\ $\eqref{ob2}\di \eqref{ob3}$  for strongly countable sets, we note that \eqref{cent} is trivial as $Y$ is now a bijection, and hence $g=_{1}11\dots$ in \eqref{cent}. 
Hence, we cannot obtain $\BOOT^{-}_{C}$ in this way, which also follows from Footnote \ref{kul} and Corollary \ref{BOEM}.  
Nonetheless, Corollary \ref{Schor2} suggest the following interesting version of item \eqref{ob2}:  
let $\MCT_{1}^{\net}$ be $\MCT_{0}^{\net}$ restricted to strongly countable index set, but `upgraded' with the existence of a \emph{modulus of convergence}.  
Similarly, item \eqref{ob1} can be reformulated based on Corollary \ref{Schor2} as:
\begin{princ}[$\BW_{1}^{[0,1]}$]
For a strongly countable set $A\subset [0,1]$, there is a sequence $(y_{n})_{n\in \N}$ in $A$ that converges to the supremum of $A$.   
\end{princ}
We could omit `to the supremum of $A$' in $\BW_{1}^{[0,1]}$ if we also assume that the sequence is not eventually constant. 
We discuss the exact properties of $\BW_{1}^{[0,1]}$ and $\MCT_{1}^{\net}$, their relation to \emph{hyperarithmetical analysis} in particular, in Remark \ref{dilf}. 

\smallskip

Finally, let $!\QFAC^{0,1}$ be $\QFAC^{0,1}$ restricted to \emph{unique} existence.  
\begin{thm}\label{fluk}
The system $\ACAo$ proves $\MCT_{1}^{\net}\asa \cocode_{1}\asa {!\QFAC^{0,1}}$ and $\MCT_{0}^{\net}\asa [\MCT_{1}^{\net}+ \CBN]$.
\end{thm}
\begin{proof}
First of all, $\cocode_{1}\leftarrow {!\QFAC^{0,1}}$ follows by applying the latter to the second condition, i.e.\ surjection on $\N$, of the definition of `bijection'.  
For the other implication, given $Y^{2}$ such that $(\forall n\in \N)(\exists! f\in 2^{\N})(Y(f, n)=0)$, define $A:=\{g\in 2^{\N}: (\exists n\in \N)(Y(g, n)=0)\} $ and define $Z:2^{\N}\di \N$
as follows: $Z(f)$ is the least $n$ such that $Y(f, n)=0$, if such there is, and zero otherwise.  By assumption, $Z$ is bijective on $A$ and $\cocode_{1}$ provides an enumeration of $A$.  
The latter sequence readily yields the choice function required by $!\QFAC^{0,1}$.   
Given $\exists^{2}$ and the usual coding, we can replace a quantifier `$(\exists! f\in \N^{\N})$' by a quantifier `$(\exists! f\in 2^{\N})$', and ${!\QFAC^{0,1}}$ follows without restrictions. 

\smallskip

Secondly, the equivalence $\MCT_{0}^{\net}\asa [\MCT_{1}^{\net}+ \CBN]$ follows from $\MCT_{1}^{\net}\asa \cocode_{1}$ combined with Theorem \ref{angel}.  
For the latter equivalence, the reverse implication is straightforward as the sequence provided by $\cocode_{1}$ for the strongly countable 
index set from $\MCT_{1}^{\net}$ allows us to perform the usual `interval-halving' proof using $\ACA_{0}$.  For the forward implication, let $\lambda w.f_{w}$ be the net as in the proof of Theorem \ref{angel} defined for a \emph{strongly} countable set $A$ and associated $Y:2^{\N}\di \N$ which is bijective on $A$.  By definition, we have $\lim_{w}f_{w}=11\dots$ and let $g^{0\di 1^{*}}$ be a modulus of convergence, i.e.\ 
\[
(\forall k\in \N)(\forall v \succeq g(k))(\overline{f_{v}}k =_{0^{*}}\overline{11\dots}k  ).
\]
Again by definition, we have $(\forall k\in \N)(\exists i<|g(k)|)(Y(g(k)(i))=k )$, i.e.\ $\cocode_{1}$ follows, and we are done.
\end{proof}
Let $\IND_{\Sigma}$ be the induction axiom for $\Sigma$-formulas, i.e.\ of the form $\varphi(n)\equiv (\exists f\in \N^{\N})(Z(f, n)=0)$.
Applying $\ECF$ from Remark \ref{ECF} shows that the base theory in Corollary \ref{wonk} is a conservative extension of $\RCA_{0}$.
\begin{cor}\label{wonk}
The system $\RCAo+\IND_{\Sigma}$ proves $\BW_{1}^{[0,1]}\asa\cocode_{1}$
\end{cor}
\begin{proof}
The reverse implication follows as in the proof of the theorem. 
For the forward direction, fix $A\subset [0,1]$ and let $Y:[0,1]\di \N$ be a bijection on $A$.  Now define $B\subset [0,1]$ as follows: $x\in B$ if and only if the binary representation $f_{x}\in 2^{\N}$ of $x$ (readily defined using $\exists^{2}$) satisfies:
\be\label{hoela}
[f_{x}=\underbrace {11\dots 11}_{\textup{length $k$}}*\langle 0\rangle * g_{0}\oplus g_{1}\oplus \dots \oplus g_{k}] \wedge (\forall i\leq k)(Y(g_{i})=i  \wedge \r(g_{i})\in A).
\ee
Note that induction as in $\IND_{\Sigma}$ is needed to establish that `$A$ is strongly countable' implies `$B$ is strongly countable'.
The sequence $(y_{n})_{n\in \N}$ provided by $\BW_{1}^{[0,1]}$ must converge to $1$ by \eqref{hoela}.  This clearly yields a sequence listing the elements of $A$.  
\end{proof}
Note that the previous proof does not go through if we omit the sequence from $\BW_{1}^{[0,1]}$.  This sequence can of course be obtained from the supremum using $\QFAC^{0,1}$, but the latter already implies $\cocode_{1}$.
Moreover, by Footnote \ref{kul}, $\MCT_{1}^{\net}$ and $\BW_{1}^{[0,1]}$ are not explosive when combined with e.g.\ the Suslin functional, in contrast to $\MCT_{0}^{\net}$ by Corollary \ref{BOEM}.

\smallskip

The following corollary deals with $\fin_{0}$, $\CBN'$, and $\NCC$ from Section \ref{bumm}. 
\begin{cor}\label{useff}
Over $\RCAo+\NCC$, we have $\cocode_{0}\asa \CBN\asa [\CBN' +\fin_{0}]$.
\end{cor}
\begin{proof}
 First of all, we may assume $(\exists^{2})$ is given, as all principles from the corollary are vacuously true given $\neg(\exists^{2})$, as in Remark \ref{flop}. 
Now, given $\NCC$, $\cocode_{0}\asa \CBN$ follows from Theorem \ref{angel} as $\NCC\di\DCA\di \cocode_{1}$ where the latter implications are from \cite{dagsamIX}*{\S3.1} and \cite{dagsamX}*{Cor.\ 3.33}.

\smallskip

Clearly, $\cocode_{0}\di [\CBN'+\fin_{0}]$ while to obtain $\BOOT_{C}^{-}$ from $\CBN'+\fin_{0}$, fix $Z$ such that $(\forall n\in \N)(\exists \textup{ at most one } g\in 2^{\N})(Z(g, n)=0)$ and define the set $A:=\{ g\in 2^{\N}: (\exists n\in \N)(Z(g, n)=0))\}$. 
This set is countable with an obvious injection defined in terms of $Z$.  Now consider 
\be\label{durks}
(\forall n\in \N)(\exists k\geq n)(\exists f\in 2^{\N})(Z(f, k)=0), 
\ee
which intuitively expresses that $A$ is infinite.  Assuming \eqref{durks} holds, apply $\NCC$ and let $g:\N\di \N$ be the resulting sequence.
Use primitive recursion to define $h(0):=0$ and $h(n+1):=g(h(n))$.  
Then $(\forall n\in \N)(\exists f\in 2^{\N} )(Z(f, h(n))=0)$ and note that this existence is \emph{unique}, by assumption on $Z$.
In the first paragraph, we mentioned $\NCC\di \cocode_{1}$, which yields $!\QFAC^{0,1}$ by the theorem.  Applying the latter to  $(\forall n\in \N)(\exists! f\in 2^{\N} )(Z(f, h(n))=0)$ yields a sequence $(f_{n})_{n\in \N}$ of \emph{distinct} elements of $A$. 
Now define the injection $\N\di A$ as mapping $n\in \N$ to $f_{n}$ and note that $A$ is strongly countable by $\CBN'$.
Apply $!\QFAC^{0,1}$ to the second part of the definition of `bijection to $\N$' to obtain a sequence listing all elements of $A$.  
Then $\BOOT_{C}^{-}$ readily follows.  If \eqref{durks} is false, then $A$ is finite in the sense of the antecedent of $\fin_{0}$, and the latter also readily yields $\BOOT^{-}_{C}$ in this case. 
\end{proof}
It is an interesting exercise to show that in the previous results, $\CBN$ can be replaced by $\CBN_{\infty}$, i.e.\ $\CBN$ with `$Y:\R\di \N$' replaced by `$Y:\R\di\N_{\infty}$' in Definition \ref{standard}.
Another possible variation is to study bijections defined as mappings $\N_{\infty}\di \R$, for which a version of \eqref{hyuk} can be used to obtain $\BOOT_{C}^{-}$.
In this light, the above results exhibit some robustness.  

\smallskip

Next, we establish Corollary \ref{surplus}, which shows that $\BW_{0}^{[0,1]}$ cannot imply $\ACA_{0}$, \emph{even if} we use a notion of set other than Definition \ref{openset}.
Here, $\MUC$ involves the \emph{intuitionistic fan functional}:
\be\tag{$\MUC$}
(\exists \Omega^{3})(\forall Y^{2})(\forall f, g\in 2^{\N})(\overline{f}\Omega(Y)=\overline{g}\Omega(Y)\di Y(f)=_{0}Y(g)),
\ee
which formalises a version of Brouwer's continuity theorem, i.e.\ that all $Y:2^{\N}\di \N$ are (uniformly) continuous.  
By Remark \ref{frorem}, $\MUC$ therefore implies $\neg(\exists^{2})$.  
As noted in \cite{kohlenbach2}*{\S3}, $\RCAo+\MUC$ is a conservative extension of $\WKL_{0}$ for $\L_{2}$-sentences.
As for Corollary \ref{crucrucor}, the following is actually a meta-result.  
\begin{cor}\label{surplus}
Regardless of the meaning of `$x\in A$', the system $\RCAo+\MUC$ proves that a countable set $A\subset [0,1]$ has a supremum.  
\end{cor}
\begin{proof}
Clearly, $\MUC$ implies $\neg(\exists^{2})$ and $\WKL$ (see \cite{simpson2}*{IV.2} for the latter). 
Now use the proof of Corollary \ref{crucrucor} to conclude that a countable set $A\subset [0,1]$ has at most one element.  
Hence, the conclusion of the corollary is vacuously true.  
\end{proof}
The previous corollary suggests that theorems about (strongly) countable sets like $\BW_{0}^{[0,1]}$ and $\MCT_{1}^{\net}$ must occupy a place between the base theory and $\WKL_{0}$ in terms of second-order consequences, regardless of the representation of sets. 
This theme is explored in (much) more detail in \cite{dagsamXI}.

\smallskip

\noindent
In conclusion, the above results lead to the following answers to (Q0)-(Q3).
\begin{enumerate}
\item[(A0)] Our notion of `countable set' yields {explosive} theorems, reaching up to $\SIX$ when combined with $\FIVE^{\omega}$ (Cor.\ \ref{BOEM}).  
Replacing `sequence' by `countable set' can result in loss of logical strength (Cor.\ \ref{Schor}). 
\item[(A1)] Our notion of `strongly countable set' does not yield explosive theorems by Footnote \ref{kul} and Corollary \ref{fluk}, in contrast to (A0).  
\item[(A2)] Nets with countable index sets can yield `explosive' convergence theorems by Corollary \ref{BOEM}, just like for uncountable index sets (see \cite{samph}*{\S3}).
\item[(A3)] The notion of (strongly) countable set yields nice, but perhaps expected, splittings in light of Theorems \ref{angel} and \ref{fluk}.
\end{enumerate}
%
We finish this section with an important remark on $\BW_{1}^{[0,1]}$ and $\MCT_{1}^{\net}$.
\begin{rem}[Hyperarithmetical analysis]\label{dilf}\rm
First of all, the notion of \emph{hyperarithmetical set} (\cite{simpson2}*{VIII.3}) gives rise the (second-order) definition of \emph{system/statement of hyperarithmetical analyis} (see e.g.\ \cite{monta2} for the exact definition), 
which includes $\Sigma_{1}^{1}$-$\textsf{CA}_{0}$ (see \cite{simpson2}*{VII.6.1}).  Montalb\`an claims in \cite{monta2} that \textsf{INDEC}, a special case of \cite{juleke}*{IV.3.3}, is the first `mathematical' statement of hyperarithmetical analysis.  The latter theorem by Jullien can be found in \cite{aardbei}*{6.3.4.(3)} and \cite{roosje}*{Lemma 10.3}.  

\smallskip

Secondly, a classical result by Kleene (see e.g.\ \cite{longmann}*{Theorem 5.4.1}) establishes that the following classes coincide: the \emph{hyperarithmetical} sets, the $\Delta_{1}^{1}$-sets, and the subsets of $\N$ computable (S1-S9) in $\exists^{2}$.  Hence, one expects a connection between hyperarithmetical analysis and extensions of $\ACAo$.  By way of example, $\ACAo+\QFAC^{0,1}$ is a conservative extension of $\Sigma_{1}^{1}$-$\textsf{CA}_{0}$ by \cite{hunterphd}*{Cor.\ 2.7}.  The latter is $\Pi_{2}^{1}$-conservative over $\ACA_{0}$ by \cite{simpson2}*{IX.4.4}. 

\smallskip

Thirdly, the monographs \cites{roosje, aardbei, juleke} are all `rather logical' in nature and $\textsf{INDEC}$ is the \emph{restriction} of a higher-order statement to countable linear orders in the sense of RM (\cite{simpson2}*{V.1.1}), i.e.\ such orders are represented by subsets of $\N$.  
In our opinion, the statements $\MCT_{1}^{[0,1]}$, $\cocode_{1}$, and $\BW_{1}^{[0,1]}$ are (much) more natural than \textsf{INDEC} as they are obtained from theorems of mainstream mathematics 
by a (similar to the case of $\textsf{INDEC}$) restriction, namely to strongly countable sets.  
Moreover, the system $\ACAo+\X$ where $\X$ is either $\MCT_{1}^{[0,1]}$, $\cocode_{1}$, or $\BW_{1}^{[0,1]}+\IND_{\Sigma}$, proves $\textsf{weak}$-$\Sigma_{1}^{1}$-$\textsf{CA}_{0}$, which is $\Sigma_{1}^{1}$-$\textsf{CA}_{0}$ with a uniqueness condition, by Theorem \ref{fluk}.  Hence, $\ACAo+\MCT_{1}^{\net}$ is a rather natural system \emph{in the range of hyperarithmetical analysis}, as it sits between $\RCAo+\textsf{weak}$-$\Sigma_{1}^{1}$-$\textsf{CA}_{0}$ and $\ACAo+\QFAC^{0,1}$.  
\end{rem}
\subsubsection{Nets with uncountable index sets}\label{dir2}
The results from Section \ref{dir1} are interesting in their own right, but also lead to improvements to results on nets with uncountable index sets from e.g.\ \cite{samnetspilot, samph}, as explored in this section.

\smallskip

With the gift of hindsight, we can now formulate the following result for $\BOOT$ and `obvious' generalisations of $\BW_{0}^{[0,1]}$ and $\MCT_{0}^{\net}$ from Theorem \ref{angel}. 
The equivalence $\MCT_{\net}^{[0,1]}\asa \BOOT$ was first proved in \cite{samph}*{\S3}; we refer to the latter or \cite{samnetspilot} for definitions of nets with index set $\N^{\N}$ in $\RCAo$. 
\begin{thm}\label{angels}
The system $\RCAo$ proves that the following are equivalent:
\begin{enumerate}
\renewcommand{\theenumi}{\roman{enumi}}
\item $\ACA_{0}$ plus: \emph{for \textbf{any} $A\subset [0,1]$ and $F:[0,1]\di [0,1]$, $\sup_{x\in A}F(x)$ exists},\label{ob12}
\item $\MCT_{\net}^{[0,1]}$, monotone convergence theorem for nets with index set $\N^{\N}$,\label{ob22}
\item $\BOOT$,\label{ob32}
\item {For any RM-closed set $C\subset [0,1]$ and $F:[0,1]\di [0,1]$, $\sup_{x\in C}F(x)$ exists}.\label{ob123}
\end{enumerate}
\end{thm}
\begin{proof}
First of all, $\eqref{ob22}\asa \eqref{ob32}$ immediately follows from \cite{samph}*{Theorem 3.7}.
Note that in case $\neg(\exists^{2})$, all functions on $\R$ are continuous by Remark \ref{frorem}.  Hence, $\BOOT$ reduces to $\ACA_{0}$ while the second part of item \eqref{ob12} becomes provable given $\ACA_{0}$ using the usual interval-halving technique.  Thus, $\eqref{ob12}\asa \eqref{ob32}$ holds \emph{in this case}.

\smallskip

Secondly, in case $(\exists^{2})$, fix some $A\subset [0,1]$ and $F:[0,1]\di [0,1]$.  Note that `$x\in A$' is decidable in light of Definition~\ref{openset}.  Now fix $(q_{n})_{n\in\N}$,  an enumeration of all rationals in $\Q$, and 
use $\BOOT$ to obtain a set $X_{0}\subset \N$ such that
\[
n\in X_{0}\asa (\exists x\in \R)(\underline{x\in A\wedge F(x)>_{\R}q_{n}}),
\]
where the underlined formula is decidable modulo $\exists^{2}$.
To define the supremum $y=\sup_{x\in A}F(x)$, define $X\subset \Q$ as the set of those $q_{n}$ with $n\in X_{0}$.  Define $[y](0)$ as $\frac{1}{2}$ if $\frac{1}{2}\in X$, and $0$ otherwise. 
For the general case, $[y](n+1)$ is $[y](n)+\frac{1}{2^{n+1}}$ if the latter is in $X$, and $[y(n)]$ otherwise.   Thus, we have $\eqref{ob12}\leftarrow \eqref{ob32}$ in general.

\smallskip

Thirdly, by the proof of \cite{samph}*{Theorem 3.7}, $\BOOT$ follows from the monotone convergence theorem for nets indexed by $2^{\N}$ or $[0,1]$, assuming $\ACA_{0}$. 
However, a net in $[0,1]$ with index set $D\subset [0,1]$ has a supremum by item \eqref{ob12}, i.e.\ we obtain $\eqref{ob12}\di \eqref{ob32}$. 
Similarly, item \eqref{ob123} guarantees the existence of the supremum of such a net, i.e.\ we also have $\eqref{ob123}\di \BOOT$. 
To obtain $\ACA_{0}$ from item \eqref{ob123}, note that for $F(x):=x$, we obtain the supremum $\sup C$ of an RM-closed set $C\subset [0,1]$.
The latter property yields $\ACA_{0}$ by (the proof of) \cite{withgusto}*{Theorem 3.8} or \cite{simpson2}*{IV.2.11}.  
The proof of $\BOOT\di \eqref{ob123}$ follows in the same way as for $\BOOT\di \eqref{ob12}$.
\end{proof}
We note that item \eqref{ob123}, formulated with RM-codes for continuous functions $F$, exists in the RM-literature, namely \cite{simpson2}*{IV.2.11.2}.
Thus, the `small' change to \emph{arbitrary} third-order objects has a massive effect.  Of course, applying $\ECF$ to $\eqref{ob123}\asa \BOOT$, one obtains the equivalence between $\ACA_{0}$ and  \cite{simpson2}*{IV.2.11.2} in accordance with the theme of \cite{samph}.
Thus, the least we can say is that our definition of set as in Definition~\ref{openset} gives rise 
to nice equivalences and splittings.
We also have the following corollary where $\SUBNET_{0}$ states that a monotone 
net in $[0,1]$ indexed by $\N^{\N}$ has a monotone sub-net with a countable index set. 
\begin{cor}
$\ACAo$ proves $[\BOOT+\QFAC^{0,1}]\asa [\BOOT_{C}^{-} +\SUBNET_{0}] $.
\end{cor}
\begin{proof}
The forward direction follows from \cite{samnetspilot}*{Cor.\ 4.6}.  Indeed, the (proof of) the latter yields that, assuming $\QFAC^{0,1}$, a monotone net in $[0,1]$ indexed by $\N^{\N}$ and converging to $x$, has a \emph{sub-sequence} that also converges to $x$.  Using $(\exists^{2})$, this sequence is readily seen to be a net with a countable index set.  
\smallskip

For the reverse direction, to obtain item \eqref{ob22} from Theorem \ref{angels}, apply $\SUBNET_{0}$ and note that $\BOOT_{C}^{-}$ is equivalent to item \eqref{ob2} from Theorem \ref{angel}. 
Moreover Corollary~\ref{Schor2} supplies a modulus of convergence, which in turn yields $\QFAC^{0,1}$ by \cite{samph}*{Cor.\ 3.17}.  
\end{proof}
On a historical note, Root, a student of E.H.\ Moore, already studied when limits from MooreÕs General Analysis \cite{moorelimit2} can be replaced by limits given by sequences \cite{rootsbloodyroots}. Thus, the idea of replacing nets by sequences goes back more than a century, but generally requires the Axiom of Choice by \cite{samnetspilot}*{Cor.\ 4.6}.


\smallskip

We finish this section with a remark on the formulation of $\BW_{0}^{[0,1]}$.  
\begin{rem}[On formalisation]\label{dagki}\rm
A valid critical question is whether $\BW_{0}^{[0,1]}$ does really formalise an instance of the `Bolzano-Weierstrass theorem for countable sets'.  As always, this kind of
question is complicated and we therefore spend some time and effort on a detailed answer. Recall that $\BW_{0}^{[0,1]}$ as in Definition~\ref{konk} states:
\begin{center}
\emph{for countable $A\subset [0,1]$ and $F:[0,1]\di [0,1]$, the supremum $\sup_{x\in A}F(x)$ exists. }
\end{center}
First of all, $\BW_{0}^{[0,1]}$ states \emph{intuitively speaking} that the supremum exists for the set $F(A)=\{y\in [0,1] : (\exists x\in A)(F(x)=_{\R}y) \}$ for countable $A\subset [0,1]$ and $F:[0,1]\di [0,1]$.  
However, it is a common theme in RM that in weak systems certain sets or functions do not exist as mathematical objects, but only in a certain `virtual' or `comparative' sense (see e.g.\ \cite{simpson2}*{p.\ 392}) or represented by a formula (like open sets in e.g.\ \eqref{morg} above).  
Now, $\BW_{0}^{[0,1]}$ is equivalent to the following: 
\begin{center}
\emph{for countable $A\subset [0,1]$ and $Y:[0,1]\di \N$ injective on $A$, the set $\{ n\in \N: (\exists x\in A)(Y(x)=n) \}$ exists}.
\end{center}
Hence, there is no hope that in general $F(A)$ as above exists as a set in $\Z_{2}^{\omega}$.  In fact, this existence would yield $\BW_{0}^{[0,1]}$ by the previous.  For this reason, the formulation of $\BW_{0}^{[0,1]}$ avoids mentioning the set $F(A)$.  Moreover, since $F(A)$ does not exist as a set in $\RCAo$, the question whether it is countable (in the sense of Definitions~\ref{openset} and~\ref{standard}) does not really make sense.  

\smallskip

Secondly, the previous paragraph notwithstanding, we \emph{can} give meaning to statements like `the set $X_{\varphi}=\{x\in [0,1]:\varphi(x)\} $ is countable', namely as 
\be\label{barf}
(\exists Y:[0,1]\di \N)(\forall x, y\in X_{\varphi})( Y(x)=_{0}Y(y)\di x=_{\R}y   ), 
\ee
where `$x\in X_{\varphi}$' is just short for `$\varphi(x)$'.  Now consider $B$ and $F(w):=f_{w}$ from the proof of Theorem \ref{angel}.  One readily proves that `the set $F(B)=\{y\in [0,1]:(\exists w^{1^{*}})(F(w)=y)\} $ is countable' in the sense of \eqref{barf}.  Indeed, let $(\sigma_{n})_{n\in \N}$ be an enumeration of the finite binary sequences and define $Y_{0}(f)$ as $n+1$ if $n$ is the unique number such that $f=_{1}\sigma_{n}*00\dots$, and $0$ otherwise.  Clearly, we have 
\be\label{barf2}
(\forall f, g\in F(B))( Y_{0}(f)=_{0}Y(g)\di f=_{1}g  ), 
\ee
i.e.\ we can restrict $\BW_{0}^{[0,1]}$ to countable sets $A\subset [0,1]$ and $F:[0,1]\di [0,1]$ where $F(A)$ is countable in the sense of \eqref{barf}.
The resulting restriction is however less elegant than $\BW_{0}^{[0,1]}$.  Moreover, Theorem \ref{crucru} explains why one wants to avoid `sets' like $X_{\varphi}$: even in the
basic case of $\N_{\infty}$, one obtains $(\exists^{2})$, which implies $\ACA_{0}$.  

\smallskip

Thirdly, we can view $F(B)$ from the proof of Theorem \ref{angel} as a \emph{subset} of a countable set, where `inclusion' is interpreted in the `comparative' second-order sense.  
Indeed, one readily proves `$F(B)\subset D$', where $D=\{f \in 2^{\N} : (\exists n\in \N)(f=_{1}\sigma_{n}*00\dots) \}$ and where $\sigma_{n}$ is the $n$-th finite binary sequence.  
The set $D$ exists and is countable (following Definition \ref{standard}) given $\ACAo$.  To be absolutely clear, `$F(B)\subset D$' means the following:
\be\label{barf3}
(\forall f\in 2^{\N})\big[ (\exists w^{1^{*}}\in B)(F(w)=_{1}f)  \di f\in D  \big].
\ee
In light of the above, we can restrict $\BW_{0}^{[0,1]}$ to $A$ and $F$ such that $F(A)\subset D$ for some countable $D\subset [0,1]$.
This formulation is however again far less elegant.   

\smallskip

In conclusion, while $F(B)$ may not exist in weak systems, we can express the countability of $F(B)$ via \eqref{barf2} or \eqref{barf3}.  
Restricting to such sets does not change the equivalences in Theorem \ref{angel}.
Thus, $\BW_{0}^{[0,1]}$ does deal with the supremum of countable sets, but the technical details can be treacherous.  
\end{rem}
\subsection{K\"onig's lemma versus lemmas by K\"onig}\label{waycool}
We study well-known eponymous lemmas by K\"onig from \cite{koning147,koning26} for (strongly) countable sets (of reals) following Definition~\ref{standard}.  We show that these lemmas imply $\NBI$ or $\cocode_{1}$ and are hence unprovable in $\Z_{2}^{\omega}$, 
but provable in $\Z_{2}^{\Omega}$, i.e.\ without additional choice axioms.   In contrast to the 
(second-order) version involving trees, our lemmas based on Definition~\ref{standard} do not imply $\ACA_{0}$.  
Nonetheless, we obtain some interesting reversals, including \emph{splittings} $\cocode_{0}\asa [\Korg_{0}+\CBN]$ over $\RCAo$, where $\Korg_{0}$ is a lemma due to K\"onig formulated using Definition \ref{standard}.
We also discuss the history of these lemmas, as this would seem illuminating in light of \cite{simpson2}*{p.\ 125, Notes for \S III.3}.  

\smallskip

First of all, let \emph{K\"onig's tree lemma} be the statement \emph{every infinite finitely branching tree has a path}.  When formulated in $\L_{2}$, the latter is equivalent to $\ACA_{0}$ by \cite{simpson2}*{III.7.2}.  
Simpson refers to \cite{koning147} as the original source for K\"onig's tree lemma in \cite{simpson2}*{p.\ 125}, but \cite{koning147} does not mention the word `tree' (i.e.\ the word `Baum' in German).  In fact, K\"onig's original (graph theoretic) lemma from \cite{koning147} is as follows.
\begin{princ}[K\"onig's infinity lemma for graphs]\label{KIL}
If a countably infinite graph $G$ can be written as countably many non-empty finite sets $E_{1}, E_{2}, \dots$ such that each point in $E_{n+1}$ is connected to a point in $E_{n}$ via an edge, then $G$ has an infinite path $a_{1}a_{2}\dots$ such that $a_{n}\in E_{n}$ for all $n\in \N$. 
\end{princ}
The original version, introduced a year earlier in \cite{koning26}, is formulated in the language of set theory as follows, in both \cite{koning26, koning 147}.  
\begin{princ}[K\"onig's infinity lemma for sets]\label{KIL2}
Given a sequence $E_{0}, E_1, \dots$ of finite non-empty sets and a binary relation $R$ such that for any $x\in E_{n+1}$, there is at least one $y\in E_{n}$ such that $yRx$.
Then there is an infinite sequence $(x_{n})_{n\in \N}$ such that for all $n\in \N$, $x_{n}\in E_{n}$ and $x_{n}Rx_{n+1}$.
\end{princ}
The names \emph{K\"onig's infinity lemma} and \emph{K\"onig's tree lemma} are used in \cite{wever} which contains a historical account of these lemmas, as well 
as the observation that they are equivalent; the formulation involving trees apparently goes back to Beth around 1955 in \cite{bethweter}, as also discussed in detail in \cite{wever}.
When formulated in set theory, K\"onig's infinity lemma is equivalent to a fragment of the Axiom of Choice (\cite{levy1}*{p.\ 298}) over $\ZF$, and is (strictly) implied by Ramsey's theorem (\cite{fortru}).

\smallskip
 
Now, since `countable subset of $\R$' is a \emph{third-order} notion, it makes sense for Simpson to study the \emph{second-order} K\"onig's tree lemma, although this begs the question what the logical strength
of K\"onig's infinity lemma(s) is.   By the below, K\"onig's infinity lemmas \emph{formulated for countable subsets of $\R$ as in Definition \ref{standard}} are not provable in $\Z_{2}^{\omega}$, in contrast to the $\L_{2}$-version of K\"onig's tree lemma of course is (provable in $\ACA_{0}$).  
Nonetheless, we can obtain nice equivalences for certain versions of these lemmas, showing in particular that the Axiom of Choice is not needed  (as everything is provable in $\Z_{2}^{\Omega}$).  

\smallskip

Next, we discuss some conceptual motivations for our study of the infinity lemmas.  The following quote by K\"onig constitutes motivation and evidence that graph theory was intended to be infinitary.      
\begin{quote}
Diese Bemerkung ist wichtig, da, wenn man sie einmal angenommen hat, nichts im Wege steht die ``Sprache der Graphen'' auch dann zu nuetzen, wenn die Mengen [\dots] nicht endlich, ja sogar von beliebig grosser Machtigkeit sind. (\cite{koning16}*{p.\ 460})
\end{quote}
The final sentence states that graphs of any cardinality can be studied in graph theory.  
Moreover, we note that K\"onig's infinity lemma is introduced in \cite{koning147} as a graph-theoretic formulation of another theorem from \cite{koning26}.  In both the French (\cite{koning26}*{\S3}) and German formulation (\cite{koning147}*{\S1}), the word `sequence' is used in the conclusion, i.e.\ an infinite path is a sequence of elements.  By contrast, the antecedent is always formulated using countable sets.  Thus, the above K\"onig's infinity lemmas are close to the original theorems `as they stand', i.e.\ without enrichment or modification.     

\smallskip

In light of these observations, we define $\Korg_{0}$ as follows where we interpret `finite set' as a set $A\subset\R$ for which there exists $Y:\R\di \N$ which is injective and bounded on $A$, i.e.\ just like for $\fin_{0}$ introduced in Section \ref{bumm}.  As is clear from its proof, Theorem \ref{kiro} (and Corollary \ref{Qutree}) still goes through if we replace item \eqref{kriol} in Principle \ref{KILreal} by `for all $n\in \N$ there exists a non-empty finite sequence listing the elements of $V_{n}$'.  This endows the below results with a certain robustness, an important concept in RM according to Montalb\'an (\cite{montahue}).
\begin{princ}[$\Korg_{0}$]\label{KILreal}
Let $G=(V, E)$ be a graph where $V=\cup_{n\in \N}V_{n}\subset \R$ and 
\begin{enumerate}
 \renewcommand{\theenumi}{\alph{enumi}}
\item the vertex set $V$ is countable,
\item each $V_{n}$ is non-empty and finite,\label{kriol}
\item each vertex in $V_{n+1}$ is connected to a vertex in $V_{n}$ via an edge in $E$.  
\end{enumerate}
Then there is a sequence $(v_{n})_{n\in \N}$ such that $v_{n}\in V_{n}$ and $(v_{n}, v_{n+1})\in E$ for $n\in \N$. 
\end{princ}
We let $\Korg_{1}$ be $\Korg_{0}$ restricted to strongly countable sets $V$.
In our opinion, the principles $\Korg_{i}$ for $i=0,1$ are therefore close to K\"onig's orginal interpretation/meaning.
We now have the following theorem. 
\begin{thm}\label{kiro}
The system $\RCAo$ proves $\Korg_{1}\di \NBI$.
\end{thm}
\begin{proof}
Suppose $\neg\NBI$ and note that $[0,1]$ is strongly countable.  The graph $G=(V, E)$ is defined as $V=[0,1]$ with incidence relation $xEy $ defined as $Y(y)=Y(x)+1 \vee Y(x)=Y(y)+1$, where $Y:[0,1]\di \N$ is the bijection provided by $\neg\NBI$.  
Define $V_{n}:=\{ x\in [0,1]:Y(x)=n\}$ and note that this collection satisfies all the requirements from $\Korg_{1}$, in particular for $y\in V_{n+1}$, we have $xEy$ for $x\in V_{n}$.
Hence, there is a path $(x_{n})_{n\in \N}$ through $G$.  Note that this sequence enumerates $[0,1]$ and use \cite{simpson2}*{II.4.9} to obtain a contradiction. 
\end{proof}
The following corollary provides a nice answer to (Q0).  We believe the use of $\CBN$ to be essential in the second part. 
\begin{cor}\label{wood}
The system $\RCAo$ proves $\Korg_{1}\asa  \cocode_{1} $ and $\cocode_{0}\asa [\Korg_{0}+\CBN]$.
The system $\RCAo+\NCC$ proves $\cocode_{0}\asa [\Korg_{0}+\CBN'+\fin_{0}]$.
\end{cor}
\begin{proof}
For the first part, similar to the proof of Theorems \ref{crucru} and \ref{angel}, we may assume $(\exists^{2})$, as the equivalence trivially holds in case of $\neg(\exists^{2})$, following Remark~\ref{flop}.
The implication $\Korg_{1}\di \cocode_{1}$ follows by repeating the proof of the theorem for $[0,1]$ replaced by a strongly countable set $A\subseteq \R$.
For the reverse implication, let $G$ be as in $\Korg_{1}$ and let $(y_{n})_{n\in \N}$ be a sequence such that $x\in G\asa (\exists n\in \N)(x=_{\R}y_{n})$ for any $x\in \R$, as provided by $\cocode_{1}$. 
We can now carry out the `usual' proof of K\"onig's tree lemma as $(\exists x\in G)A(x)$ is equivalent to $(\exists n\in \N)A(y_{n})$.    
Indeed, expressions like `an element of $G$ has infinitely many successors in $G$' are now arithmetical, which $\ACA_{0}$ can handle as in the usual proof of K\"onig's tree lemma (from $\L_{2}$). 

\smallskip

For the second part, we may assume $(\exists^{2})$ as in the previous paragraph.  
Since $\Korg_{0}\di \Korg_{1}\di \cocode_{1}$, the reverse implication is immediate. 
For the forward implication, $\CBN$ readily follows by using $\exists^{2}$ to `trim' duplicate reals from the sequence provided by $\cocode_{0}.$
The usual proof of K\"onig's lemma (in $\ACA_{0}$) yields $\Korg_{0}$ after converting the countable set into a sequence using $\cocode_{0}$.
The third part is now immediate by Corollary \ref{useff}.
\end{proof}
By the following, $\Korg_{1}$ is provable without using countable choice.  Note that (the proof of) Corollary \ref{surplus} implies that the exact definition of set also does not matter for the second part. 
\begin{cor}\label{Qutree}
The system $\RCAo+\DCA$ proves $ \Korg_{1}$, while $\RCAo+\Korg_{1}$ cannot prove $\ACA_{0}$.
\end{cor}
\begin{proof}
It is show in \cite{dagsamX}*{\S3} that $\QFAC^{0,1}\di \DCA\di \cocode_{1}$ over $\RCAo$. 
Since $\ECF$ converts $\QFAC^{0,1}$ to $\QFAC^{0, 0}$, the former does not imply $\ACA_{0}$.
\end{proof}
We now turn to Principle \ref{KIL2}, formalised as follows.  A binary relation $R$ on reals is given by a characteristic function $F_{R}:\R^{2}\di \{0,1\}$, i.e.\ $xRy\equiv F_{R}(x, y)=1$. 
\begin{princ}[$\Korg_{2}$]\label{KIL2real}
Let $(E_{n})_{n\in \N}$ be a sequence of sets in $\R$ and let $R$ a binary relation on reals such that for all $n\in\N$ we have:
\begin{itemize}
\item the set $E_{n}$ is finite and non-empty,
\item for any $x\in E_{n+1}$, there is at least one $y\in E_{n}$ such that $yRx$.
\end{itemize}
Then there is a sequence $(x_{n})_{n\in \N}$ such that for all $n\in \N$, $x_{n}\in E_{n}$ and $x_{n}Rx_{n+1}$.
\end{princ}
As expected, we have $\Korg_{2}\di \cocode_{1}$ but no reversal is known.  Moreover, the use of extra induction in the proof of $\Korg_{2}$ seems necessary.  
Let $\Pi_{2}^{1}$-$\IND^{\omega}$ be the induction axiom for $\Pi_{2}^{1}$-formulas with arbitrary parameters.
\begin{thm}\label{wood2}
 $\RCAo$ proves $\Korg_{2}\di \cocode_{1}$ and $\RCAo+\ACA_{0}+\Pi_{2}^{1}\textup{-}\IND^{\omega}+\QFAC^{0,1}$ proves $\Korg_{2}$.  
\end{thm}
\begin{proof}
To obtain $\Korg_{2}\di\cocode_{1}$, let $A\subset \R$ be strongly countable as witnessed by the bijection $Y:A\di \N$. 
Define $E_{n}:=\{x\in A: Y(x)= n\}$ and note that this set is non-empty and finite for any $n\in \N$.  In fact, $E_{n}$ has exactly one element by definition.  
Define the binary relation $yRx$ as $Y(x)+1=_{0} Y(y)$ and note that for $y\in E_{n+1}$, there is (exactly one) $x\in E_{n}$ such that $yRx$.
Hence, $\Korg_{2}$ provides $(x_{n})_{n\in \N}$ such that $x_{n}\in E_{n} $ and $x_{n+1}Rx_{n}$ for any $n\in \N$.
By definition, $Y(x_{n})=n$ for any $n\in \N$ and $\cocode_{1}$ follows.  

\smallskip

For the second part, we shall prove the following for all $n\in \N$:
\be\label{truf}
(\exists w^{1^{*}})\big[|w|=n \wedge (\forall i<|w|)(w(i)\in E_{i})\wedge (\forall j<|w|-1)(w(j)Rw(j+1))   \big].
\ee
Now apply $\QFAC^{0,1}$ to obtain $(w_{n})_{n\in \N}$.  With the latter sequence, one readily\footnote{Let $\sigma \in T$ hold in case $\sigma$ is an initial segment of $w_{n}$ for some $n\in \N$.} builds a finitely branching tree (in the sense of second-order RM).  Now use K\"onig's tree lemma (provided by $\ACA_{0}$; see \cite{simpson2}*{III.7.2}) to obtain the path required by $\Korg_{2}$.  To prove \eqref{truf} for all $n\in \N$, consider the formula $\varphi(n)$ defined as
\[
(\forall x\in E_{n})(\exists w^{1^{*}})\left[\begin{array}{c} |w|=n+1 \wedge (\forall i<|w|)(w(i)\in E_{i})\wedge x=w(n)\\
\wedge  (\forall j<|w|-1)(w(j)Rw(j+1)  )\end{array}\right].
\]
Now use $\Pi_{2}^{1}$-$\IND^{\omega}$ to prove $(\forall n\in \N)\varphi(n)$, which implies \eqref{truf} for all $n\in \N$.
\end{proof}
By \cite{dagsamX}*{Theorem 3.1}, $\Z_{2}^{\omega}+\QFAC^{0,1}$ cannot prove $\NIN$, which is established via 
a model \textbf{P} of the former in which the latter is false. It can be shown that this model (or a variation thereof) satisfies the induction axiom as in the previous theorem (and more). 
As a result, $\Korg_{2}$ cannot imply $\NIN$ (or $\cocode_{0}$ for that matter). 

\smallskip

As a preliminary conclusion, our results concerning K\"onig's various lemmas are somewhat more complicated than those in Section \ref{LP}.  
On one hand, Corollary \ref{wood} provides the nice equivalence $\Korg_{1}\asa \cocode_{1}$, but we do not know how to obtain an equivalence between $\Korg_{0}$ and $\cocode_{0}$.  
On the other hand, the splittingd in Corollary \ref{wood} and the results in Theorem \ref{wood2} are of course quite nice.  

\smallskip

Finally, the restriction of K\"onig's tree lemma to \emph{binary} trees yields the second Big Five system, called $\WKL_{0}$.  We briefly discuss the latter and associated equivalences.  

\smallskip

Now, $\WKL_{0}$ is equivalent to the Heine-Borel theorem for \emph{sequences} of intervals covering $[0,1]$ (\cite{simpson2}*{IV.1}).  
The following version is closer to Borel's original version from \cite{opborrelen2}, as discussed in Remark \ref{horgku}.
\begin{princ}[$\HBC_{0}$]
For countable $A\subset \R^{2}$ with $(\forall x\in [0,1])(\exists (a, b)\in A)(x\in (a, b))$, there are $(a_{0}, b_{0}), \dots (a_{k}, b_{k})\in A$ with $(\forall x\in [0,1])(\exists i\leq k)(x\in (a_{i},b_{i} ))$.
\end{princ}
We now have the following theorem.  Similar results hold for e.g.\ Vitali's covering theorem (see e.g.\ \cite{simpson2}*{X.1}).
\begin{thm}\label{Schor3}
The system $\Z_{2}^{\omega}+\QFAC^{0,1}$ cannot prove $\HBC_{0}$, the latter does not imply $\WKL_{0}$ over $\RCAo$, and the same for any sentence not provable in $\RCA_{0}$.  
\end{thm}
\begin{proof}
For the first part, it is shown in \cite{dagsamX} that $\HBC_{0}\di \NIN$ and $\Z_{2}^{\omega}+\QFAC^{0,1}\not\vdash \NIN$.
For the second part, note that $\RCAo+\HBC_{0}\vdash \WKL_{0}$ is converted to $\RCA_{0}\vdash \WKL_{0}$ as $\HBC_{0}$ is vacuously true under $\ECF$ following the proof of Corollary \ref{crucrucor}.
Indeed, $\ECF$ replaces all third-order objects by (continuous by definition) RM-codes, meaning that countable sets are interpreted as finite sets.
\end{proof}
We finish this section with a conceptual remark.  
\begin{rem}[Similar results]\rm
Since $\RCAo+\MUC$ is a conservative extension of $\WKL_{0}$, there is not much sense in proving a result like Corollary \ref{surplus} here.  
However, note that $\ACA_{0}$ follows from the statement that \emph{any separably closed set in $[0,1]$ has the Heine-Borel property} (see \cite{hirstrm2001} for details).  
Similar to the proof of Corollary \ref{surplus}, we have that regardless of the meaning of `$x\in A$', the system $\RCAo+\MUC$ proves that any separably closed set in $[0,1]$ has 
the Heine-Borel property \emph{formulated with countable collections of intervals} $A\subset \R^{2}$.  Hence, $\HBC_{0}$ generalised to separably closed sets does not imply $\ACA_{0}$.  
\end{rem}

\subsection{On theorems from the RM zoo}\label{krazy}
We study theorems from the RM zoo formulated using (strongly) countable set (of reals) as in Definition \ref{standard}.
We provide detailed results for $\ADS$ and sketch the results for $\CAC$ and $\RT_{2}^{2}$.  
We assume basic familiarity with the RM zoo and the aforementioned principles, although $\ADS$ is introduced below in Principle \ref{ADS}.  

\smallskip

In particular, we show that $\Z_{2}^{\omega}$ cannot prove the higher-order versions of $\ADS$, $\CAC$, and $\RT_{2}^{2}$ formulated using Definition \ref{standard}, while $\Z_{2}^{\Omega}$ of course can prove these higher-order versions, i.e.\ the Axiom of Choice is not needed.  
Similar to the previous, we have the splitting $\cocode_{0}\asa [\ADS_{0}+\CBN]$ over $\RCAo$, where $\ADS_{0}$ is $\ADS$ formulated using Definition \ref{standard} as in Principle \ref{rivals}.
The same holds for the higher-order versions of $\CAC$ and $\RT_{2}^{2}$ based on Definition \ref{standard}, i.e.\ the RM zoo is a lot more `tame' formulated in third-order arithmetic.

\smallskip

First of all, as to motivation, the word `countable' and variations appears about one hundred times in \cite{dsliceke}, Hirschfeldt's monograph that provides 
a partial overview of the RM zoo.   Countable infinity does indeed take centre stage, as is clear from Hirschfeldt's quote in Section \ref{detail}.
As it happens, this quote is preceded in \cite{dsliceke} by:
\begin{quote}
The work of G\"odel and others has shown that mathematics, like everything else, is built on sand. As Borges reminds us, this fact should not keep us from building, and building boldly. 
However, it also behooves us to understand the nature of our sand.
\end{quote}
While we do not agree with Hirschfeldt's foundational claims regarding G\"odel, we share his sentiment regarding the necessary nature of the study of the foundations of mathematics.  
Thus, it behooves us to study the logical strength of theorems from the RM zoo formulated using the `real' definition of countable set.  

\smallskip

The previous paragraph constitutes general motivation, but particular theorems come with `extra' motivation.  We single out fragments of \emph{Ramsey's theorem} 
as Ramsey himself in \cite{keihardrammen}, the original source of `Ramsey's theorem', formulates the infinite version of Ramsey's theorem using `infinite sets' and \textbf{not} using sequences.
Moreover, versions of the \emph{Rival-Sands theorem} from \cite{rivalsanders} are apparently equivalent to $\ADS$ and $\RT_{2}^{2}$ (see \cite{rivalsanders2, rivalsanders3}).  
The following quote by Rival-Sands strongly suggests their work is also formulated using `infinite (countable) sets' and \textbf{not} sequences. 
\begin{quote}
Recently, M.\ Gavalec and P.\ Vojtas have pointed out to us that the natural
analogue of our Theorem 1 holds for graphs of regular cardinality $\kappa$. (\cite{rivalsanders}*{p.\ 396})
\end{quote}
We could obtain a version of Theorem \ref{ladel} and corollaries for the various second-order versions of the Rival-Sands theorem.
On a related note, the topic of \cite{ericwaszonegoeie} is the (Weihrauch degree) study of $\ADS$ and variations involving (second-order) sets rather than sequences in the consequent.  
Thus, our idea of studying $\ADS$ based on Definition \ref{standard} is definitely \emph{in the same spirit}.  

\smallskip

Secondly, we now formulate the \emph{ascending-desending sequence principle} from \cite{dsliceke}*{Def.\ 9.1}, which is the following $\L_{2}$-sentence.
\begin{princ}[$\ADS$]\label{ADS}
Every infinite linear order has an infinite ascending or descending sequence.
\end{princ}
Countable linear orders are represented by subsets of $\N$ (see e.g.\ \cite{simpson2}*{V.1.1}) in RM, and we now study what happens if we adopt the definition of (strongly) `countable set' as in Definition \ref{standard}

\smallskip

Of course, we use the usual definition of `linear order' $(X, \preceq_{X})$; we shall assume that $X\subseteq \R$ or $X\subset \N^{\N}$, as this already guarantees that the associated
third-order version of $\ADS$ is not provable in $\Z_{2}^{\omega}$.
If the set $X$ is (strongly) countable, then we say that  $(X, \preceq_{X})$ is (strongly) countable.  An infinite ascending sequence $(x_{n})_{n\in \N}$ in $X$ satisfies $x_{n}\prec x_{n+1}$ for any $n\in\N$, and similar for descending sequences.
We say that $(X, \prec_{X})$ is infinite if for any $n\in \N$ there are pairwise distinct $x_{0}, \dots, x_{n} \in X$.  

\smallskip

With the above in mind, Definition \ref{standard} in particular, we make the following definition in $\RCAo$.
We note that $\ADS_{0}$ is provable in $\Z_{2}^{\Omega}$ by \cite{dagsamX}*{Figure 1}, where the latter system does 
not involve the Axiom of Choice.  
\begin{princ}[$\ADS_{0}$]\label{rivals}
Every countable and infinite linear order has an infinite ascending or descending sequence.
\end{princ}
Let $\ADS_{1}$ be $\ADS_{0}$ restricted to \emph{strongly} countable sets.  
Recall that $\IND_{\Sigma}$ is the induction axiom for $\Sigma$-formulas, i.e.\ of the form $\varphi(n)\equiv (\exists f\in \N^{\N})(Z(f, n)=0)$.

\smallskip

Thirdly, we obtain some results about $\ADS_{0}$, beginning with the following. 
\begin{thm}\label{ladel}
The system $\RCAo+\IND_{\Sigma}$ proves $\ADS_{0}\di \NBI$.
\end{thm}
\begin{proof}
Let $Y:[0,1]\di \N$ be a bijection, i.e.\ we have $\neg\NBI$.   Note that we have access to $(\exists^{2})$, which allows us to convert reals into binary representation.  
Define the set $X=\{w^{1^{*}}: (\forall i<|w|)(Y(w(i))=i \wedge w(i)\in [0,1] \}$, which is readily seen to be countable. 
Define the binary relation $\preceq_{X}$ on $X$ as follows: $w\preceq_{X} v $ if $|w|\leq |v|$.  
Trivially, this relation is transitive and connex.  Since $Y$ is an injection, this relation is also antisymmetric, and hence $(X, \preceq_{X})$ is a linear order.  
Since $Y$ is also a bijection, we have $(\forall n\in \N)(\exists w^{1^{*}})(|w|=n+1\wedge w\in X)$, which has an obvious proof using $\IND_{\Sigma}$.  
Hence, $(X, \preceq_{X})$ is an \emph{infinite} linear order.  Applying $\ADS_{0}$, there is an infinite ascending or descending sequence.
Now, this sequence cannot be descending and let $(w_{n})_{n\in \N}$ be such that $w_{n}\prec w_{n+1}$ for all $n\in \N$.  
By definition, we have $Y(w_{n}(n))=n$, i.e.\ the sequence $(w_{n}(n))_{n\in \N}$ lists all reals in $[0,1]$.  By \cite{simpson2}*{II.4.9}, there is $y\in [0,1]$ not in this sequence.  
But then for $n_{0}:= Y(y)$, we have $Y(w_{n_{0}}(n_{0}))=n_{0}=Y(y)$, a contradiction, and $\NBI$ follows. 
\end{proof}
\begin{cor}\label{prev1}
The system $\RCAo+\IND_{\Sigma}$ proves $\ADS_{0}\di \cocode_{1}$, while $\ADS_{0}$ does not imply $\ADS$ over $\RCAo$.
\end{cor}
\begin{proof}
For the first part, repeat the proof of the theorem with $[0,1]$ replaced by a strongly countable set $A\subset \R$.  
For the second part, $\ADS_{0}$ is vacuously true under $\ECF$, i.e.\ applying the latter to $\RCAo+\ADS_{0}\vdash \ADS$ leads to a contradiction. 
\end{proof}
We do not know how to prove $\ADS_{1}\di \cocode_{1}$ or $\ADS_{0}\di \cocode_{0}$.  
We do have the following corollary.
\begin{cor}\label{flor}
The system $\RCAo+\IND_{\Sigma}$ proves $\cocode_{0}\asa [\ADS_{0}+\CBN]$.
The system $\RCAo+\IND_{\Sigma}+\NCC$ proves $\cocode_{0}\asa [\ADS_{0}+\CBN'+\fin_{0}]$.
\end{cor}
\begin{proof}
The first part follows from Corollary \ref{prev1} and (the proof of) Corollary \ref{wood}.  Note in particular that $(\exists^{2})\di \ACA_{0}\di \ADS$.
The second part follows from Corollary \ref{useff}.
\end{proof}
For the following corollary, a linear order $(X, \preceq_{X})$ is \emph{stable} if every element either has only finitely many $\prec_{X}$-
predecessors or has only finitely many $\prec_{X}$-successors.  The linear order from Theorem \ref{ladel} is clearly stable.  
\begin{cor}
Theorem \ref{ladel} and corollaries also hold for $\textup{\textsf{SADS}}_{0}$, i.e.\ $\ADS_{0}$ restricted to stable linear orders. 
\end{cor}
The following corollary should be contrasted to the fact that $\WKL_{0}\not\vdash \ADS$ (\cite{dsliceke}). 
Recall that $\RCAo+\MUC$ is a conservative extension of $\WKL_{0}$ (see \cite{kohlenbach2}*{\S3}).
\begin{cor}\label{surplus2}
Regardless of the meaning of `$x\in A$', $\RCAo+\MUC$ proves that a countable infinite linear order in $[0,1]$ has an infinite ascending or descending sequence.
\end{cor}
\begin{proof}
As in the proof of Corollary \ref{surplus}.
\end{proof}
Fourth, we discuss similar results for related theorems from the RM-zoo.  Such results hold for e.g.\ the \emph{chain-antichain principle} $\CAC$ which expresses that \emph{every infinite partial order has an infinite chain or antichain} in $\L_{2}$ (\cite{dsliceke}*{Def.~9.2}).
Indeed, $(X, \preceq_{X})$ from the proof of Theorem \ref{ladel} is a countable (in the sense of Definition~\ref{standard}) and infinite partial order without infinite antichains.  Moreover, $\RT_{2}^{2}$ implies $\CAC$ via an elementary argument (\cite{dsliceke}*{p.\ 144})
that carries over to the third-order versions.  We repeat that Ramsey formulates the infinite version of Ramsey's theorem  in \cite{keihardrammen}*{p.\ 264} using `infinite sets' (Ramsey uses `classes' rather than `sets').  
Hence, at the very least, we should study Ramsey's theorem formulated using (strongly) countable sets rather than sequences.

\smallskip

In conclusion, $\RT_{2}^{2}$ and $\CAC$ formulated with the definition of `countable set' as in Definition \ref{standard}, is not provable in $\Z_{2}^{\omega}$.  Nonetheless, since they are both provable in $\ACA_{0}$, 
$\CAC$ and $\RT_{2}^{2}$ yield splittings as in Corollary \ref{flor}, i.e.\ the RM zoo is easily tamed by introducing Definition \ref{standard}.


\subsection{Countable sets in mathematics and logic}\label{bauer}
As noted in Section \ref{detail}, we do not claim that the definition of countable sets via injections/bijections to $\N$ constitutes the `standard' or `mainstream' one.  
We have studied these notions in higher-order RM since they yield interesting results.  In this section, we discuss some related results in the grand(er) scheme of things, 
as well as an argument for the study of Kunen's definition of countable set based on injections to $\N$.  We believe the results in this section to provide some context for the results in this paper. 

\smallskip

First of all, we list textbooks in which `countable sets' are defined via sequences.  
\begin{itemize}
\item The textbook \emph{Introductory Real Analysis} by Kolmogorov and Fomin (\cite{kollen}).
\item The textbook \emph{Calculus} by Spivak (\cite{spiva}).  
\item Bishop's textbook on \emph{Constructive Analysis} (\cite{bish1}).
\end{itemize}
In particular, the definition of countable set base via sequences appears to be the usual definition in the setting of elementary calculus and real analysis where the general notion of `cardinality' is not needed or developed.

\smallskip

Secondly, it is well-known that `disasters' can happen in the absence of the Axiom of Choice ($\AC$ for short), by which it is meant that \emph{many beautiful theorems are no longer provable in the absence of $\AC$}, as discussed in  \cite{heerlijkheid}*{Preface}.  We point out that such disasters already happen for rather `mundane' topics like finite or countable sets, like e.g.\ the fact that $\R$ is not the countable union of countable sets, or basic cardinal arithmetic (see \cite{heerlijkheid}*{\S4}).  Nonetheless, the principles studied in this paper, especially $\cocode_{i}$ for $i=0,1$, are all provable in $\Z_{2}^{\Omega}$ and weaker systems, i.e.\ without $\AC$.  As it happens, the author and Dag Normann study the relationship between the \emph{countable union theorem} and $\cocode_{i}$ for $i=0,1$ in \cite{dagsamXI}*{\S3.2}.  

\smallskip 

Thirdly, \emph{constructive mathematics} is usually qualified as mathematics based on intuitionistic logic with some appropriate extra `semi-constructive' axioms (see e.g. \cite{brich}).  
For instance, Bishop's aforementioned \emph{constructive analysis} additionally assumes the axiom of countable choice (and other axioms).  
The field \emph{constructive RM} seeks to develop RM over a constructive base theory (see e.g.\ \cite{ishi1, hadie}).  A result relevant to this paper may be found
in \cite{bauer1}, which essentially shows that $\NIN$ can be false in certain approaches to constructive mathematics.   
Another related result is in \cite{cblem}, showing that the Cantor-Bernstein theorem implies the law of excluded middle. 
A well-known aspect of constructive mathematics (which also shows up in classical RM) is
that classically equivalent notions are no longer equivalent in a constructive setting.  A relevant example pertaining to countability is \emph{subcountability}, but we will content ourselves with pointing the interested reader to e.g.\ \cite{raco,macol, scow1} and the references therein.  

\smallskip

Fourth, we provide an argument for the study of Kunen's definition of countable set based on injections to $\N$.  In \cite{dagsamX}, Dag Normann and the author study the RM of the principles $\NIN$ and $\NBI$, also introduced in Section \ref{bintro}.  We identify a number of third-order principles that do not mention the notion `countable set' (based on bijections, injections, or enumerations) explicitly, yet all imply $\NIN$.  
In fact, it is quite hard to find a natural principle that does not mention `countable set' explicitly, implies $\NBI$, and does not imply $\NIN$; a somewhat natural example can be found in \cite{samNEO2}.
We list two examples of natural third-order theorems of ordinary mathematics that do not mention `countable set' in any way, but do imply $\NIN$.  
\begin{itemize}
\item Arzel\`a's convergence theorem \emph{for the Riemann integral} (1885; \cite{arse2}).
\item Jordan's decomposition theorem (1881; \cite{jordel}).  
\end{itemize}
The proof that the first item implies $\NIN$ may be found in  \cite{dagsamX}*{\S3.1.2}.
That the second item implies $\NIN$ has not been published yet.  

\begin{ack}\rm
I thank Anil Nerode and Dag Normann for their helpful suggestions.
My research was supported by the John Templeton Foundation via the grant \emph{a new dawn of intuitionism} with ID 60842 and by the \emph{Deutsche Forschungsgemeinschaft} via the DFG grant SA3418/1-1.  Opinions expressed in this paper do not necessarily reflect those of the John Templeton Foundation.  
The results in Section~\ref{krazy} go back to the stimulating BIRS workshop (19w5111) on Reverse Mathematics at CMO, Oaxaca, Mexico in Sept.\ 2019.  
I thank the anonymous referees for their many suggestions that have greatly improved this paper, esp.\ Section \ref{bauer}.  
\end{ack}

\bye